\documentclass[leqno,11pt]{amsart}
\usepackage{amsmath,amssymb}
\usepackage{color}
\usepackage{comment}

\usepackage{amsfonts}      
\usepackage{amsthm}         
\usepackage{bbding}         
\usepackage{bm}             
\usepackage{graphicx}   
\usepackage{fancyvrb}       
			    
\usepackage{dcolumn}        
\usepackage{booktabs}       
\usepackage{paralist}       
\usepackage{xcolor}        

\usepackage{amssymb}
\usepackage{units}
\usepackage{mathtools}
\usepackage[colorlinks=true,linkcolor=blue]{hyperref}%

\allowdisplaybreaks

\theoremstyle{plain}
\newtheorem{thm}{Theorem}[section]
\newtheorem{lemma}[thm]{Lemma}
\newtheorem{prop}[thm]{Proposition}
\newtheorem{cor}[thm]{Corollary}

\theoremstyle{definition}
\newtheorem{defn}[thm]{Definition}
\newtheorem{rem}[thm]{Remark}

\newcommand{\Rea}{\mathbb{R}}
\newcommand{\Nat}{\mathbb{N}}

\renewcommand{\d}{\,{\rm d}}
\newcommand{\eps}{\varepsilon}
\newcommand{\emb}{\hookrightarrow}
\newcommand{\INT}{{\rm Int}}

\DeclareMathOperator{\sgn}{\textrm{sgn}}
\DeclareMathOperator*{\esssup}{ess\,sup}

\newcommand{\abs}[1]{\left|{#1}\right|}
\newcommand{\csup}[1]{\underset{#1}{\rm {csup}}}

\newcommand{\medint}{-\kern  -,435cm\int}
\newcommand{\medintinrigo}{-\kern  -,315cm\int}
\newcommand{\medelle}{-\kern  -,235cm L}
\newcommand{\medellenrigo}{-\kern  -,180cm L}

\newtoks\by
\newtoks\paper
\newtoks\book
\newtoks\jour
\newtoks\yr
\newtoks\pages
\newtoks\vol
\newtoks\publ
\newtoks\eds
\newtoks\proc
\newtoks\no
\def\ota{{\hbox{???}}}
\def\cLear{\by=\ota\paper=\ota\book=\ota\jour=\ota\yr=\ota
\pages=\ota\vol=\ota\publ=\ota}
\def\endpaper{\the\by, \textit{\the\paper},
{\the\jour} \textbf{\the\vol} (\the\yr), \the\pages.\cLear}
\def\endbook{\the\by, \textit{\the\book}, \the\publ.\cLear}
\def\endprep{\the\by, \textit{\the\paper}, \the\jour.\cLear}
\def\endproc{\the\by, \textit{\the\paper}, \the\publ, \the\pages.\cLear}

\numberwithin{equation}{section}

\newcommand{\norm}[1]{{\left\vert\kern-0.25ex\left\vert\kern-0.25ex\left\vert
#1
    \right\vert\kern-0.25ex\right\vert\kern-0.25ex\right\vert}}

\usepackage[normalem]{ulem}
\usepackage{soul}
\usepackage{cancel}
\newcommand\Del[1]{{\color{red}\ifmmode\cancel{#1}\else\sout{#1}\fi}}

\usepackage[
	textsize=tiny,
	textwidth=3.2cm,
	colorinlistoftodos,
	backgroundcolor=teal!30!white,
	linecolor=teal!50!white
	]{todonotes}

\begin{document}

\date{\today}

\title[Optimal function spaces]{Optimal function spaces and Sobolev embeddings}

\author {David Kub\'i\v cek}

\address{David Kub\'i\v cek, Department of Mathematical Analysis\\
Faculty of Mathematics and Physics\\
Charles University\\
Sokolovsk\'a~83\\
186~75 Praha~8\\
Czech Republic} 

\email{kubicek@karlin.mff.cuni.cz}

\begin{abstract}
    We establish equivalence between the boundedness of specific supremum operators and the optimality of function spaces in Sobolev embeddings acting on domains in ambient Euclidean space with a prescribed isoperimetric behavior. Our approach is based on exploiting known relations between higher-order Sobolev embeddings and isoperimetric inequalities. We provide an explicit way to compute both the optimal domain norm and the optimal target norm in a Sobolev embedding. Finally, we apply our results to higher-order Sobolev embeddings on John domains and on domains from the Maz'ya classes. Furthermore, our results are partially applicable to embeddings involving product probability spaces.
\end{abstract}

\maketitle


\section{Introduction}\label{intro}
Sobolev embeddings have been closely studied for several decades, ever since the period in which they were pioneered in the works of Sobolev \cite{So1936}, Gagliardo \cite{Ga1958},  and Nirenberg \cite{Ni1959}. Parallel to this research, isoperimetric inequalities were explored by De Giorgi \cite{Gi} and Federer and Fleming \cite{FeFl}. Some time later, Sobolev embeddings have been connected with isoperimetric inequalities in the works of Maz'ya \cite{Ma1960, Ma1961}. The study of the mentioned connection has been rapidly developing ever since, and in 2015 a comprehensive paper \cite{CiPiSl} presented a unified approach to the topic, and, moreover, a technique suitable for finding optimal target spaces in the Sobolev embeddings on fairly general underlying measure spaces.

Optimality of function spaces with respect to various classes of function spaces has been under scrutiny for a long time, and a wide variety of various points of view has been taken. This research is clearly motivated by the pure fact that one wants to have the results as sharp as possible in order to get strongest possible outcome in applications. On the other hand, the approach necessarily leads to interesting technical obstacles. For one thing, an optimal space does not have to exist. It has been noticed that when optimality is restricted to classes of spaces such as Lebesgue spaces, Lorentz spaces, Orlicz spaces, etc., then there is no guarantee that the optimal object will be found. On the other hand, the research in the last 25 years shows that when one settles for working in the (wider) environment of the so-called rearrangement-invariant spaces (called also symmetric spaces in some older literature), then it is almost certain that an optimal space (on either domain or target position) exists. This approach however brings new difficulties, the pivotal one being the fact that the optimal spaces are often described in a very implicit way, as they involve operations such as taking the associate space, or taking suprema over all equidistributed functions etc.

Consequently, a lot of effort has been spent on making the description of optimal function spaces as `explicit' as possible. These activities paved way to discoveries of certain interesting phenomena such as, for instance, a surprising connection between optimality of function spaces, their interpolation properties, and boundedness of certain nonlinear operators. Applications of these advances to sharp Sobolev embeddings exist, but exclusively restricted to domains with a Lipschitz boundary. However, as is widely known, it is desirable to have Sobolev-type embeddings for domains with much `worse' isoperimetric behavior such as domains with cusps, and also domains endowed with product probability measures whose typical example is the Gauss measure. These domains are typically characterized by their isoperimetric profile, or isoperimetric function, denoted $I$. Let us recall, for instance, that the best possible isoperimetric behavior occurs for John domains (whose subclass is that of Lipschitz domains) when $I(t)=t^{\frac{1}{n'}}$, where $n$ is the dimension and $n'$ is the conjugate index, while the domains with power cusps have $I(t)=t^{\alpha}$ with $\alpha<\frac{1}{n'}$, and the Gauss measure leads to $I(t)\approx t\sqrt{\log \frac 1t}$ near zero.

The aim of this paper is to obtain explicit description of optimal rearrangement-invariant partner spaces in Sobolev embeddings acting on domains with a general isoperimetric profile. We will achieve this by combining the ideas concerning optimality of spaces for integral operators with the approach to Sobolev embeddings via the isoperimetric profile of an underlying domain. We begin by carefully constructing relevant nonlinear operators and prove their connection to the question of optimality of function spaces. More precisely, we shall exploit the equivalence of Sobolev embeddings and isoperimetric inequalities in order to connect Sobolev embeddings on a domain having the isoperimetric profile $I$ with boundedness of two supremum operators, $S_I$ and $T_I$, defined as $$S_If(t)=\frac{1}{I(t)}\sup_{0<s\leq t}I(s)f^*(s)\quad\text{and}\quad T_If(t)=\frac{I(t)}{t}\sup_{t\leq s<1}\frac{s}{I(s)}f^*(s)$$ for $f\in\mathcal M_+(0, 1)$ and $t\in (0, 1)$, in a manner similar to that in \cite{KePi2006} and \cite{KePi2009}. We will prove in particular that the boundedness of $S_I$ is related to the optimality of domain spaces, and $T_I$ corresponds to that of  target spaces. Boundedness of operators similar to $T_I$ and $S_I$ was previously studied for example in \cite{KePhPi2014}, but only for the case when $I$  is a power. There, the operators were considered in the context of Orlicz $L_A$ and Gamma $\Gamma_{p, \phi}$. Hence, once the equivalence between the optimality of spaces and boundedness of $T_I$ and $S_I$ is established, it will enable us to recover and considerably extend these results.

The fair isoperimetric generality will bring some unavoidable restrictions on the function $I$. The main setting in the paper will be that $I$ is a quasiconcave function as this allows us to work with both operators $S_I$ and $T_I$ the way we need. There are two main conditions concerning $I$ appearing throughout the paper. The first one reads as 
\begin{align}\label{eq:00001}
\int_0^t \frac{I(s)}{s}\d s\lesssim I(t),\quad t\in (0, 1).
\end{align} 
We will see that, for example, the isoperimetric functions of product probability spaces do have this property. This condition allows us to state our first principal result. Given a rearrangement-invariant domain space $X$, we denote by $Y_X$ the smallest rearrangement-invariant target space in the relevant Sobolev embedding. We similarly denote by $X_Y$ the largest rearrangement-invariant domain space when the target space $Y$ is fixed. Detailed definitions are given below.

\begin{thm}\label{thm:Norm1} 
Let $I$ be a quasiconcave function satisfying $\eqref{eq:00001}$ and let $X$ be an~r.i.~space. Then
\begin{equation*}
    \|f\|_{Y_X}\approx\sup_{\|S_Ig\|_{X'}\leq 1}\int_0^1-I(t)g^*(t)\d f^*(t)+\|f\|_1,\quad f\in\mathcal M_+(0, 1).
\end{equation*}
\end{thm} 

Furthermore, using this condition we show the boundedness of the operator $T_I$ on the associate spaces of the optimal target spaces. However, in order to prove an analogous result for the operator $S_I$, we further require a condition similar to $\eqref{eq:00001}$,
\begin{align}\label{eq:00002}
\int_0^t \frac{\d s}{I(s)}\lesssim \frac{t}{I(t)},\quad t\in (0, 1),
\end{align} which we will refer to as the \emph{average property} of $I$. It is easy to see that this condition enforces integrability of $\frac{1}{I}$ and is thus more restraining than $\eqref{eq:00001}$ in this way.

The main theorem of the paper connects the boundedness of supremum operators and optimality of spaces with respect to the operator $H_I$, which is defined as 
$$H_If(t)=\int_t^1\frac{f(s)}{I(s)}\d s$$ for $f\in\mathcal M_+(0, 1)$ and $t\in (0, 1)$. In this theorem, however, we will need a few more technical assumptions. In particular, we will require $I$ belong to a certain class, denoted $\mathcal Q$, which will be defined below. Then, the second main result reads as

\begin{thm}\label{thm:main}
Let $I\in \mathcal Q$. Then an r.i.~space $X$ is the optimal domain space under the map $H_I$ for some r.i.~space $Y$ if and only if $S_I$ is bounded on $X'$. In that case,
\begin{equation}\label{eq:Optimality0}
\|f\|_{Y_X}\approx \left\|\frac{I(t)}{t}(f^{**}(t)-f^*(t))\right\|_X+\|f\|_1,\quad f\in\mathcal M_+(0, 1).
\end{equation}

Vice versa, an r.i.~space $Y$ is the optimal target space under the map $H_I$ for some r.i.~space $X$ if and only if $T_I$ is bounded on $Y'$. In that case,
\begin{equation}\label{eq:Optimality01}
\|f\|_{X_Y}\approx \left\|\int_t^1 \frac{f^*(s)}{I(s)}\d s\right\|_{Y},\quad f\in\mathcal M_+(0, 1).
\end{equation}
\end{thm}

The principal achievement of Theorem~\ref{thm:main} is the explicitness of formulae for norms governing the optimal spaces. This fact becomes particularly useful in applications when particular function spaces are in question and their optimal partner spaces are to be nailed down.

Let us recall that the quantity $f^{**}(t)-f^*(t)$, which appears in~\eqref{eq:Optimality0}, is known to measure, in a certain sense, the symmetrized oscillation of a function. It has been regularly surfacing in various parts of analysis since early 1980's when it was first used in order to introduce the weak version of $L^{\infty}$ and some other related function classes in~\cite{Ben:81}. The importance of this quantity partly follows from the classical inequality
\begin{equation*}
    f^{**}(t)-f^{*}(t) \lesssim (\nabla f)^{**}(t)t^{\frac{1}{n}},
\end{equation*}
which holds, with a dimensional constant, for every smooth function $f$ with compact support on $\mathbb R^n$ and every positive $t$. The same quantity was utilized several times in various contexts, for example in order to obtain a sharper form of a Sobolev embedding in~\cite{Bas:03}, see also~\cite{Mil:04,Pus:05,Mar:06}, and it is of importance in the theory of classical Lorentz spaces and their applications, see~\cite{Car:05,Car:08}. Its indispensability was beautifully explained by Sinnamon in~\cite[Section~3.5]{Sin:07}. On the other hand, one has to be careful when working with it since it does not possess any reasonable general monotonicity property, and structures built upon it are notoriously known to lack linearity and normability.

Utilizing Theorem~\ref{thm:main}, we prove the intimate relation between Sobolev embeddings and the action of supremum operators for Maz'ya class of domains $\mathcal J_\alpha$ for $\alpha\in \left[\frac{1}{n'}, 1\right)$. This will be done in the last section of the paper. However, if one can prove the equivalence of the Sobolev embeddings and isoperimetric inequalities for a function $I$ of one's choice satisfying the assumptions of Theorem~\ref{thm:main}, one then gets the mentioned equivalence even on domains with isoperimetric profile of $I$. In our case, we use the equivalence of Sobolev embeddings and the isoperimetric inequalities for the domains from the Maz'ya class $\mathcal J_\alpha, \alpha\in\left[\frac{1}{n'}, 1\right)$. In particular, we recover the result for Lipschitz or, more generally, John domains from \cite{KePi2009}.

As the results rely on the boundedness of the Hardy-type operator $f\mapsto \frac{1}{t}\int_0^t f(s)\d s$, they are not generally applicable to $\mathcal J_1$, as the function $\frac{1}{I}$ is not integrable near zero. We will see that the norm of the optimal target space cannot be expressed as in $\eqref{eq:Optimality0}$ for certain functions $I\in\mathcal J_1$.\\


The paper is organised as follows. The second and preliminary section covers background results and is divided into four subsections. First we recall the notion of the nonincreasing rearrangement and so-called rearrangement-invariant Banach function spaces, which will form our main framework. It will, however, be necessary to delve a bit deeper and work with quasi-Banach function spaces, such as the weak Lebesgue space $L^{1, \infty}$. In the second subsection we introduce a few properties of functions and name typical examples of functions possessing them. The third subsection covers Sobolev spaces built upon rearrangement-invariant spaces and their connection to isoperimetric inequalities. The fourth subsection is devoted to the interpolation theory and, in particular, to the theory of the $K$-functional.

In the third section we study the basic properties of the supremum operators $S_I$ and $T_I$ and. Here, we will start with a very general function $I$, requiring it only to be nondecreasing, and we will demand certain other properties as we work through the section. It is mainly due to the operator $S_I$, and the relevant Marcinkiewicz type space $m_I$, that we are forced to work with quasi-Banach function spaces. However, the condition $\eqref{eq:00002}$ characterizes when the $m_I$ is in fact a Banach space, and implies subadditivity of $S_I$. The end of the section then calls into play the condition $\eqref{eq:00001}$ and we show its equivalence to two other statements which will play a crucial role in the main, fourth section.

The fourth section finally connects optimal spaces with the boundedness of supremum operators. In its first subsection we present an alternative description of the associate optimal norm via a functional which admits boundedness of the operator $S_I$. This is in turn used to describe the optimal target norm. Starting with the second subsection, we find an alternative description of the optimal target norm under the assumption of boundedness $f\mapsto \frac{1}{t}\int_0^t f(s)\d s$. The culmination of the section is then the third subsection which establishes equivalence between optimal spaces and boundedness of $T_I$ or $S_I$ on their associate spaces.

The fifth section connects the results we obtained in the previous sections with Sobolev embeddings. We show that the product probability spaces satisfy the main condition $\eqref{eq:00001}$. We then use Theorem~\ref{thm:main} to demonstrate a few examples both on the domain part and the space part.


\section{Preliminaries}\label{preli}

Throughout the paper, we write $A\lesssim B$ if $A$ is dominated by a constant multiple of $B$, independent of all quantities involved; these quantities will usually be evident from the context. By $A\approx B$ we mean that both $A\lesssim B$ and $A\gtrsim B$.

\subsection{Rearrangement invariant spaces}

In this section we recall some basics from the theory of rearrangement invariant spaces. Proofs and more details can be found in first two chapters of \cite{BS}.

Let $(\Omega, \mu)$ be a nonatomic $\sigma$-finite measure space. We set 
\begin{align*}
\mathcal M(\Omega, \mu)&=\{f\colon\Omega\to [-\infty, \infty]\colon f\mbox{ is $\mu$-measurable in $\Omega$}\},\\
\mathcal M_+(\Omega, \mu)&=\{f\in\mathcal M(\Omega, \mu)\colon f\geq 0\}
\end{align*}
and
\begin{align*}
\mathcal M_0(\Omega, \mu)&=\{f\in\mathcal M(\Omega, \mu)\colon f\mbox{ is finite $\mu$-a.e. in $\Omega$}\}. \hspace{1cm}
\end{align*}

We will often, for brevity, write only $\mathcal M(\Omega)$ if there is no risk of confusion, and similarly for the other two sets. When $\Omega\subset\Rea$ is measurable, unless stated otherwise, we will consider the Lebesgue measure which we will be denoted by $\lambda$. When considering a unit interval $(0, 1)$, which will be of particular interest to us, we will simply write $\mathcal M(0, 1)$.

Given $f\in\mathcal M(\Omega)$, we define its \emph{nonincreasing rearrangement}, denoted $f^*$, by
\begin{equation*}\label{eq:NonincreasingRearrangement}
f^*(t)=\inf\{\lambda\geq 0: \mu(\{\abs{f}>\lambda\})\leq t\},\quad t\in [0, \infty).
\end{equation*}

The nonincreasing rearrangement is monotone i.e. $$\abs{f}\leq \abs{g}\ \ \mu\text{-a.e. implies}\ \  f^*\leq g^*.$$
 
\color{black}
The nonincreasing rearrangement satisfies \emph{Hardy-Littlewood inequality}
\begin{equation}\label{eq:HLIneq}
\int_\Omega \abs{f(x)g(x)}{\rm d}\mu(x) \leq \int_0^\infty f^*(t)g^*(t)\d t,\quad f, g\in\mathcal M(\Omega). 
\end{equation} 
The operation $f\mapsto f^*$ is not subadditive and only satisfies the weaker condition \begin{equation}\label{eq:Almostsubadditivity}(f+g)^*(t_1+t_2)\leq f^*(t_1)+g^*(t_2),\quad t_1, t_2>0.\end{equation} It turns out that passing to the so called \emph{maximal nonincreasing rearrangement}, defined by
\begin{equation*}\label{eq:MaximalNonincreasingRearrangement}
f^{**}(t)=\frac{1}{t}\int_0^t f^*(s)\d s,\quad f\in\mathcal M(\Omega), t\in (0, \infty),
\end{equation*} we gain subadditivity. To be precise, it holds that $(f+g)^{**}\leq f^{**}+g^{**}$ for $f, g\in\mathcal M(\Omega)$. The maximal nonincreasing rearrangement is also nonincreasing and we have $f^*\leq f^{**}$.

Having defined the nonincreasing rearrangement, we are ready to define the notion of a rearrangement-invariant (r.i.) Banach function norm.

\begin{defn}
A mapping $\rho\colon \mathcal M_+(0, 1)\to [0, \infty]$ is called \emph{rearrangement invariant Banach function norm}, or r.i.~norm for short, if it satisfies the following conditions:
\begin{enumerate}[(P1)]
	\item $\rho(f)=0\iff f=0 \ $a.e., \\$\rho(af)=a\rho(f), \quad f\in\mathcal M_+(0, 1), a\geq 0$,\\ $\rho(f+g)\leq\rho(f)+\rho(g),\quad f, g\in\mathcal M_+(0, 1)$,
	\item $f\leq g\ \mbox{a.e.}\implies \rho(f)\leq\rho(g), \quad f, g\in\mathcal M_+(0, 1)$,
	\item $f_n\nearrow f\  $a.e. $\implies \rho(f_n)\nearrow \rho(f),\quad f, f_n\in\mathcal M_+(0, 1), n\in\Nat$,
	\item $\rho(\chi_{(0, 1)})<\infty$,
	\item $\int_0^1 f(t)\d t\lesssim \rho(f),\quad f\in\mathcal M_+(0, 1)$,
	\item $\rho(f)=\rho(f^*),\quad f\in\mathcal M_+(0, 1)$.
\end{enumerate}
\end{defn}

Sometimes we will work with a functional which is not a norm, but still satisfies rearrangement invariance -- so-called r.i.~quasinorm.
\begin{defn}
A mapping $\rho\colon \mathcal M_+(0, 1)\to [0, \infty]$ is called \emph{rearrangement-invariant quasi-Banach function norm}, or r.i.q.~norm for short, if it satisfies conditions (P2), (P3), (P4), (P6) and
\begin{enumerate}[(Q1)]
	\item $\rho(f)=0\iff f=0 \ $a.e., \\$\rho(af)=a\rho(f), \quad f\in\mathcal M_+(0, 1), a\geq 0$,\\ $\exists C\geq 1\ \rho(f+g)\leq C\rho(f)+\rho(g),\quad f, g\in\mathcal M_+(0, 1)$.
\end{enumerate}
\end{defn}

When $\rho$ is an r.i.q.~norm, we define its \emph{associate functional}, $\rho'$, by 
\begin{equation*}\label{eq:AssociateNorm}
\rho'(f)=\sup_{\rho(g)\leq 1} \int_0^1 f(t)g(t)\d t,\quad f\in\mathcal M_+(0, 1).
\end{equation*}
An immediate consequence of the definition of the associate functional is \emph{Hölder's inequality}
\begin{equation*}\label{eq:HölderInequality}
\int_0^1 f(t)g(t)\d t\leq \rho(f)\rho'(g),\quad f, g\in\mathcal M_+(0, 1),
\end{equation*} under the convention $0\cdot\infty =0$ on the right-hand side.

When $\rho$ is an r.i.~norm, its associate norm $\rho'$ is an r.i.~norm as well and obeys the \emph{principle of duality}, that is,
\begin{equation*}\label{eq:PrincipleOfDuality}
\rho''\coloneqq (\rho')'=\rho.
\end{equation*}
Given $f, g\in\mathcal M_+(0, 1)$, \emph{Hardy's lemma} asserts that 
\begin{equation}\label{eq:HardyLemma}
f^{**}(t)\leq g^{**}(t),\ t\in (0, 1) \implies \int_0^1 f^*(t)h(t)\d t\leq \int_0^1 g^*(t)h(t)\d t
\end{equation} for every $h\in\mathcal M_+(0, 1)$ nonincreasing. An important consequence of Hardy's lemma and the principle of duality is the \emph{Hardy-Littlewood-Pólya (HLP) principle}, which reads as follows:
\begin{equation}\label{eq:HLP}
f^{**}(t)\leq g^{**}(t),\ t\in (0, 1) \implies \rho(f)\leq\rho(g)
\end{equation} whenever $\rho$ is an r.i.~norm.

For an r.i.q~norm $\rho$ we further define $X=X(\rho)$ as a collection of all $f\in\mathcal M(0, 1)$ such that $\rho(\abs{f})<\infty$. Equipping $X$ with a quasinorm defined by
$\|f\|_X\coloneqq \rho(\abs{f})$ for $f\in X$, we immediately see that $X=(X, \|\cdot\|_X)$ is a quasinormed linear space. By \cite[Corollary 3.8]{NePe}, $(X, \|\cdot\|_X)$ is a complete metric space, and spaces defined in this manner are called \emph{rearrangement-invariant quasi-Banach function spaces} or, as we will often say for brevity, \emph{r.i.q.~spaces}. If $\rho$ is in fact an r.i.~norm, the space $X=X(\rho)$ is called \emph{rearrangement-invariant Banach function space} or briefly r.i.~space. By $X'$ we denote the space corresponding to $\rho'$ and call it the \emph{associate space} of $X$. 

By $X_b$ we denote the closure of simple functions in the space $X$.

The \emph{fundamental function} corresponding to an r.i.q.~space $X$, $\varphi_X$, is defined by
\begin{equation*}\label{eq:FundamentalFunction}
\varphi_X(t)=\|\chi_{(0, t)}\|_X,\quad t\in (0, 1).
\end{equation*} 
The fundamental function of r.i.~space $X$ satisfies $$\varphi_X(t)\cdot\varphi_{X'}(t)=t,\quad t\in (0, 1).$$

A consequence to this is that if $X$ is an r.i.~space, then $\varphi_X$ is a \emph{quasiconcave} function, that is $$t\mapsto \varphi_X(t)\mbox{ is nondecreasing, }\quad t\mapsto \frac{\varphi_X(t)}{t}\ \mbox{is nonincreasing}$$ and $\varphi_X(t)=0 \iff t=0$.

Whenever $\varphi$ is a quasiconcave function, there exists its \emph{least concave majorant}, say $\Tilde{\varphi}$, which satisfies $$\frac{1}{2}\Tilde{\varphi}(t)\leq \varphi(t)\leq \Tilde{\varphi}(t),\quad t\in(0, 1).$$ Furthermore, every r.i.~space $X$ can be equivalently renormed so that $\varphi_X$ is a concave function -- we will from now on assume that every r.i.~space has been renormed in this fashion.

Let now $X$ and $Y$ be two r.i.q.~spaces. We write $X\subset Y$ if $f\in X\Rightarrow f\in Y$. When $T$ is an operator on $\mathcal M_+(0, 1)$, we say that $T$ is \emph{bounded} from $X$ to $Y$ if 
\begin{equation*}\label{eq:BoundedOperator}
\|Tf\|_Y\lesssim\|f\|_X,\quad f\in X,
\end{equation*} and denote this fact by $T\colon X\to Y$. If $X=Y$, we say that $T$ is bounded on $X$. In the particular case when $T=Id$, an inclusion operator, we have \cite[Corollary 3.10]{NePe}
\begin{equation*}\label{eq:ContinuousInclusion}
X\subset Y\iff Id\colon X\to Y.
\end{equation*} In other words, inclusions between r.i.q.~spaces are always continuous. The fact that $Id\colon X\to Y$ will be denoted as $X\emb Y$.

We say that an operator $T'$ on $\mathcal M_+(0, 1)$ is an \emph{associate operator} of $T$ if
\begin{equation*}
\int_0^1 (Tf)(t)g(t)\d t = \int_0^1 f(t) (T'g)(t)\d t,\quad f, g\in\mathcal M_+(0, 1).
\end{equation*} For two r.i.~spaces $X$ and $Y$ one sees that 
\begin{equation}\label{eq:associateoperatorss}
    T\colon X\to Y \iff T'\colon Y'\to X'
\end{equation} and $\|T\|=\|T'\|$. 

For $s>0$ the dilation operator $E_s$ defined for $f\in\mathcal M(0, 1)$ by 
\begin{equation*}
(E_s f)(t)= f\left(\frac{t}{s}\right)\chi_{(0, \min\{s, 1\})}(t)\quad t\in (0, 1).
\end{equation*} It is proved in \cite[Chapter 3, Proposition 5.11]{BS} for r.i.~spaces and, more generally, in \cite[Theorem 3.23]{NePe} for r.i.q.~spaces, that $E_s$ is bounded on every r.i.q.~space.

\begin{defn}
Let $T$ be an operator on $\mathcal M_+(0, 1)$ and $X$ and $Y$ be r.i.~spaces. We say that $Y$ is an \emph{optimal target space} for $X$ under the mapping $T$, if $T\colon X\to Y$ and for every r.i.~space $Z$ the following implication holds:
\begin{equation*}
T\colon X\to Z\implies Y\emb Z.
\end{equation*} 

Vice versa, we say that $X$ is an \emph{optimal domain space} for $Y$ under the mapping $T$, if $T\colon X\to Y$ and for every r.i.~space $Z$ the following implication holds:
\begin{equation*}
T\colon Z\to Y\implies Z\emb X.
\end{equation*}
\end{defn}

In the main section we will use the \emph{level function} which is closely related to the nonincreasing rearrangement.

\begin{defn}
Let $f\in M_+(0, 1)$. Then the level function of $f$, denoted $f^\circ$, is the derivative of the least concave majorant of $t\mapsto\int_0^t f(s)\d s, t\in (0, 1)$.
\end{defn}

Take note that, as $f^*$ is nonincreasing, $t\mapsto \int_0^t f^*(s)\d s$ is a concave function, and so 
\begin{equation}\label{eq:levelfunctionandrearrangement}
\int_0^t f^\circ (s)\d s\leq \int_0^t f^*(s)\d s,\quad f\in\mathcal M_+(0, 1), t\in (0, 1).
\end{equation}

G. Sinnamon proved in \cite[Corollary 2.4]{GS} that 
\begin{equation}\label{eq:LevelFunction}
\|f^\circ\|_{X'}=\|f\|_{X_d'},\quad f\in\mathcal M_+(0, 1).
\end{equation}
Here, $\|\cdot\|_{X_d'}$ refers to the \emph{down dual associate norm} of an r.i.~space $X$, which is defined by 
\begin{equation*}
\|f\|_{X_d'}=\sup_{\|g\|_X\leq1}\int_0^1 f(t)g^*(t)\d t,\quad f\in\mathcal M_+(0, 1).
\end{equation*} Evidently $\|f\|_{X_d'}\leq \|f\|_{X'}$ for every $f\in\mathcal M_+(0, 1)$. Observe, however, that $\|f\|_{X_d'}= \|f\|_{X'}$ whenever $f$ is nonincreasing.

Classical examples of r.i.~spaces would be \emph{Lebesgue} $L^p(0, 1)$ spaces, where $1\leq p\leq \infty$, whose norm is defined by
\begin{equation}\label{eq:LpNorm1}
\|f\|_p=\left(\int_0^1\abs{f(t)}^p \d t\right)^\frac{1}{p}
\end{equation} if $1\leq p<\infty$ and 
\begin{equation*}\label{eq:LpNorm2}
\|f\|_\infty=\esssup \abs{f}.
\end{equation*} We use the convention that $\frac{1}{\infty}=0\cdot\infty=0$. Defining $p'=\frac{p}{p-1}$ for $p\in [1, \infty]$, one has $(L^p)'=L^{p'}$.

Classical examples of r.i.q.~spaces, which are not normed nor embedded in $L^1$, are Lebesgue's $L^p$ spaces with $p\in (0, 1)$, whose quasinorm is defined as in $\eqref{eq:LpNorm1}$, or the weak Lebesgue space $L^{1, \infty}$ with a quasinorm defined as 
\begin{equation*}
    \|f\|_{1, \infty}=\sup_{0<t<1} tf^*(t).
\end{equation*}
There is the largest and the smallest r.i.~space. To be precise, it holds true that 
\begin{equation*}\label{eq:SmallestLargestSpace}
L^\infty\emb X\emb L^1
\end{equation*} for every r.i.~space $X$.

One possible generalization of Lebesgue spaces, which we will be particularly interested in, are \emph{Lorentz-Zygmund spaces}. Let $p, q\in [1, \infty]$ and $\beta, \gamma\in\Rea$. The Lorentz-Zygmund spaces $L^{p, q, \beta, \gamma}$ are defined by the functional 
\begin{equation*}\label{eq:LZ}
    \|f\|_{p, q, \beta, \gamma}=\left\|t^{\frac{1}{p}-\frac{1}{q}}\ell_1^\beta(t)\ell_2^\gamma(t) f^*(t)\right\|_q,\quad f \in\mathcal{ M_+}(0, 1).
\end{equation*} Here, $\ell_1(t)=1+\abs{\log(t)}$ and $\ell_2(t)=1+\log\ell_1(t)$. When $\gamma=0$, we write $L^{p,q,\beta}$ and if $\beta=\gamma=0$, we simply write $L^{p, q}$. It is known the functional $\|\cdot\|_{p,q,\beta}$ is equivalent to an r.i.~norm if and only if $p=q=1, \beta\geq0$, or if $p\in (1, \infty), q\in [1, \infty]$, or if $p=\infty, q\in [1, \infty), \beta+\frac{1}{q}<0$, or if $p=q=\infty, \beta\leq 0$. We will, when working with these spaces, assume that the parameters satisfy one of the mentioned conditions.

Let us finally define the function spaces which will be used abundantly throughout the paper.
\begin{defn}
Let $I\colon (0, 1)\to (0, 1)$ be a nondecreasing function. We introduce two functionals defined on $\mathcal M_+(0, 1)$ with values in $[0, \infty]$ by
\begin{align*}
\|f\|_{m_I}&\coloneqq \sup_{0<t<1}I(s)f^*(s),\\
\|f\|_{\Lambda_I}&\coloneqq \int_0^1 \frac{I(s)}{s}f^*(t)\d s.
\end{align*}
We further denote $m_I\coloneqq\{f\in\mathcal M_+(0, 1)\colon \|f\|_{m_I}<\infty\}$. Analogously we define the space $\Lambda_I$. 
\end{defn}

To simplify notation, by the symbol $\widetilde I$ we mean a function $\widetilde I(t)=\frac{t}{I(t)}, t\in (0, 1)$.


\subsection{Properties of isoperimetric functions}
In this subsection we list some properties of functions, which we will use throughout the paper. For the remainder of the subsection fix $I\colon(0, 1)\to (0, 1)$  nondecreasing.

\begin{defn}
 We say that $I$ satisfies $\Delta_2$ condition if $$I(2t)\approx I(t),\quad t\in \left(0, \frac{1}{2}\right),$$ and denote this fact as $I\in\Delta_2$.
\end{defn}

Let us now discuss the average property $\eqref{eq:00002}$. First observe that, as $I$ is nondecreasing, the average property implies that the function $I$ is equivalent to a quasiconcave function. Classical examples of functions satisfying the average property are the polynomials $t\mapsto t^\alpha, t\in (0, 1)$ for $\alpha\in (0, 1)$. Functions which do not possess this property include for example $t\mapsto t$ or $t\mapsto t\sqrt{\log\frac{2}{t}}$ for $t\in (0, 1)$ or any function $I$ such that $\frac{1}{I}$ is not integrable near zero for that matter. 

 We leave the discussion concerning $\eqref{eq:00001}$ to the fifth section where we present examples of functions satisfying this property.

 In the fourth section we will use the following condition
 \begin{equation}\label{eq:BoundedMaximalOp}
     \int_t^1\frac{I(s)}{s^2}\d s\lesssim \frac{1}{t}\int_0^t\frac{I(s)}{s}\d s,\quad t\in (0, 1).
 \end{equation}
Note that if $I$ is additionally quasiconcave and satisfies $\eqref{eq:00001}$, condition $\eqref{eq:BoundedMaximalOp}$ is equivalent to
 \begin{equation*}
    \int_t^1\frac{I(s)}{s^2}\d s\lesssim \frac{I(t)}{t},\quad t\in (0, 1).
\end{equation*} 

 The classical examples of a function satisfying $\eqref{eq:BoundedMaximalOp}$ are $I(t)=t^\alpha, \alpha\in (0, 1)$. Functions which do not satisfy this condition include $I(t)=t\log^\alpha\frac{2}{t}, \alpha\in\left[0, \frac{1}{2}\right]$. The condition~\eqref{eq:BoundedMaximalOp} will become particularly handy in connection with boundedness of the maximal nonincreasing rearrangement on suitable function spaces.

The next definition specifies a certain class of quasiconcave functions that will be of use to us in the second part of the fourth section.
 \begin{defn}[Class $\mathcal Q$]\label{def:Conditions}
     Let $I\colon (0, 1)\to (0, 1)$ be a quasiconcave bijection. Let 
 \begin{equation}\label{eq:Condition1}
     c=\sup\left\{\lambda\geq 0: \lambda\left(\frac{I(t)}{t^2}-1\right)\leq \int_t^1\frac{I(s)}{s^3}\d s\quad\text{for every $t\in (0, 1)$}\right\}.
 \end{equation} 
We say that $I\in \mathcal Q$ if $I$ satisfies $\eqref{eq:00001}$, $\eqref{eq:00002}$, $\eqref{eq:BoundedMaximalOp}$ and $(1-c)d\leq c$,
where $d$ denotes the smallest positive number such that 
 \begin{equation*}
     \int_t^1 \frac{I(s)}{s^2}\d s \leq d \frac{I(t)}{t}\quad\text{for every $t\in (0, 1)$.}
 \end{equation*}   
 \end{defn}
 
 \begin{rem}\label{rem:constC}
     It can be shown that for the constant $c$ from $\eqref{eq:Condition1}$ we have $c\in \left[\frac{1}{2}, 1\right)$. This can be seen by showing that the function $$F(t)=\frac{I(t)}{t^2}-1-\int_t^1\frac{I(s)}{s^3}\d s,\quad t\in (0, 1],$$ is nonincreasing and using the fact that if a quasiconcave function is continuous at zero, then it is absolutely continuous and $I'(t)\leq \frac{I(t)}{t}$ at points at which the derivative exists.
 \end{rem}

 The examples of a functions which belong to class $\mathcal Q$ are $I(t)=t^\alpha$ for $\alpha\in (0, 1)$.


\subsection{Sobolev spaces over r.i.~spaces and isoperimetric function}
Let $\Omega\subset\Rea^n$ be a domain, that is, a connected open set. We equip $\Omega$ with a finite measure $\mu$ which is absolutely continuous with respect to the Lebesgue measure with density $\omega$. More precisely, $${\rm d}\mu(x)=\omega(x){\rm d}x,$$ where $\omega$ is a Borel measurable function satisfying $\omega(x)>0$ for a.e. $x\in\Omega$. Thus, the measure of an arbitrary measurable set $E\subset\Omega$ is given by $$\mu(E)=\int_E \omega(x){\rm d} x.$$ Throughout the paper we will assume, for simplicity, that $\mu$ is normalized in such a way that $\mu(\Omega)=1$. We now recall the definition of the perimeter of a set with respect to our space $(\Omega, \mu)$ and the isoperimetric function.

\begin{defn} Let $E\subset \Rea^n$ be measurable. We define the \emph{perimeter} of $E$ in $(\Omega, \mu)$ by 
$$P_\mu(E, \Omega)=\int_{\Omega\cap \partial^M E}\omega (x)\d \mathcal H^{n-1}(x).$$
\end{defn} Here, $\mathcal H^{n-1}$ stands for the $n-1$ dimensional Hausdorff measure on $\Rea^{n}$ and $\partial ^M E$ denotes the essential boundary of $E$ in the sense of the geometric measure theory \cite{Ma11, Zi}.

\begin{defn}
The \emph{isoperimetric function} of $(\Omega, \mu)$ is a mapping $I_{\Omega, \mu}\colon [0, 1]\to [0, \infty]$ defined by
$$I_{\Omega, \mu}(t)=\inf\left\{P_\mu (E, \Omega)\colon E\subset \Omega, t\leq \mu(E)\leq\frac{1}{2}\right\}\quad\mbox{for }t\in\left[0, \frac{1}{2}\right]$$
and $I_{\Omega, \mu}(t)=I_{\Omega, \mu}(1-t)$ for $t\in (\frac{1}{2}, 1]$.
\end{defn} An easy consequence of this definition is the \emph{isoperimetric inequality} $$I_{\Omega, \mu}(\mu(E))\leq P_\mu(E, \Omega),\quad E\subset \Omega\mbox{ is measurable.}$$

It is evident from the definition that $I_{\Omega, \mu}$ is a nondecreasing function on $[0, \frac{1}{2}]$. Further, by \cite[Proposition 4.1]{CiPiSl}, we know that $I_{\Omega, \mu}(t)\lesssim t^\frac{1}{n'}$ for $t$ sufficiently small. 

Given an r.i.~space $X$, we define $X(\Omega)=X(\Omega, \mu)$ as the collection of all $u\in\mathcal M(\Omega)$ such that $$\|u\|_{X(\Omega)}\coloneqq \|u^*\|_X$$ is finite. The functional $\|\cdot\|_{X(\Omega)}$ defines a norm on $X(\Omega)$. The space $X(\Omega)$ endowed with this norm is also called rearrangement-invariant space, and the space $X$ is called its \emph{representation space}.

The space $X'(\Omega)$ is then defined analogously via $\|\cdot\|_{X'}$.

Throughout the paper we will, for the most part, not distinguish between $X(\Omega)$ and its representation space, as it will be evident whether we work in $X(\Omega)$ or in $X$.

Let $m\in\Nat$ and $X(\Omega, \mu)$ be an r.i.~space. We define the \emph{m-th order Sobolev space} $V^m X(\Omega, \mu)$ as 
\begin{equation*}
\begin{split}
V^m X(\Omega, \mu)=\{u\colon u\ \mbox{is $m$-times weakly differentiable in $\Omega$}\\ \mbox{and } \abs{\nabla^m u}\in X(\Omega, \mu)\}.
\end{split}
\end{equation*}

The results of \cite{CiPiSl} do not require one to work exactly with $I_{\Omega, \mu}$. It suffices to have a lower bound in terms of a nondecreasing function. To be precise, we work with a nondecreasing function $I\colon [0, 1]\to [0, \infty)$ satisfying $I_{\Omega, \mu}(t)\geq cI(ct), t\in [0, \frac{1}{2}]$ for some $c>0$. 
In view of \cite[Proposition 4.2]{CiPiSl}, it is natural to assume that $I(t)\gtrsim t, t\in (0, 1)$, as this guarantees that $V^1 L^1(\Omega)\subset L^1(\Omega)$ and, consequently, that $V^1 X(\Omega)\subset L^1(\Omega)$ for every r.i.~space $X$.

We continue by introducing a pair of integral operators, $R_I$ and $H_I$, on $\mathcal M_+(0, 1)$ which are defined by 
\begin{equation*}
R_If(t)=\frac{1}{I(t)}\int_0^t f(s)\d s,\quad t\in(0, 1),
\end{equation*}
and
\begin{equation*}
H_If(t)=\int_t^1\frac{f(s)}{I(s)}\d s,\quad t\in (0, 1).
\end{equation*} Further, for $m\in\Nat$ we set
\begin{equation*}
R_I^m=\underbrace{R_I\circ \ldots\circ R_I}_{m\text{-times}} \quad\text{and }\quad H_I^m=\underbrace{H_I\circ \ldots\circ H_I}_{m\text{-times}}.
\end{equation*} Using Fubini's theorem we see that $R_I$ and $H_I$ are mutually associate. Hence, $R_I^m$ and $H_I^m$ are also mutually associate for every $m\in\Nat$. 

The operator $G_I$ is then defined by
\begin{equation*}\label{eq:OpGI}
    G_I f(t)=\sup_{t\leq s<1}R_I f^*(s),\quad f\in\mathcal M_+(0, 1), t\in (0, 1).
\end{equation*}

Therefore, for every $f\in\mathcal M_+(0, 1)$, $G_I f$ is a nonincreasing function and $R_If\leq R_I f^*\leq G_I f$ and so $(R_I f)^*\leq G_I f$.

It holds true that 
\begin{equation}\label{eq:RJeGvX}
\left\|G_I f\right\|_X\approx \|R_I f^*\|_X,\quad f\in\mathcal M_+(0, 1),
\end{equation} whenever $X$ is an r.i.~space \cite[Theorem 9.5]{CiPiSl}.

Unless stated otherwise, by $Y_X$ we will mean the optimal target space of $X$ under the mapping $H_I$, and by $X_Y$ we mean the optimal domain space of $Y$ under the mapping $H_I$ (if it exists). By the symbol $Y_X'$ we understand $(Y_X)'$ and the symbol $Y_{X_Z}$ stands for $Y_{(X_Z)}$.

The existence of the optimal target space $Y_X$ is justified in \cite[Proposition 8.3]{CiPiSl}. We will leave the question of the optimal domains to the beginning of the Section~\ref{ch:OS}.


\subsection{Interpolation theory}
\begin{defn}
Let $X_0$ and $X_1$ be quasi-Banach spaces. We say that $(X_0, X_1)$ is a \emph{compatible couple} of quasi-Banach spaces if there exists a Hausdorff topological vector space $H$ such that $X_0\emb H$ and $X_1\emb H$.
\end{defn}

We now recall few results from the interpolation theory that we will need at some proofs. Let us note that these theorems are stated in \cite[Chapter 5]{BS} in the context of Banach spaces. However, it is not hard to see that extending the results together with relevant definitions over quasi-Banach spaces does not create any problems and their proofs would only need minor, if any, modifications. Let us also note that by \cite[Theorem 3.4]{NePe} for every r.i.q.~space $X$ we have $X\emb \mathcal M_0$, where $\mathcal M_0$ is equipped with the (metrizable) topology of convergence in measure on the sets of finite measure. Consequently, any two r.i.q.~spaces form a compatible couple.

\begin{defn}
Let $(X_0, X_1)$ be a compatible couple of quasi-Banach spaces. We define the $K$-functional on $X_0+X_1$ by
$$K(f, t, X_0, X_1)=\inf\{\|g\|_{X_0}+t\|h\|_{X_1}\colon f=g+h, g\in X_0, h\in X_1\}, \quad t\in (0, \infty).$$
\end{defn}

The next theorem will be of use to us, especially when combined with Theorem~\ref{thm:KFunctionalDensity}.
\begin{thm}
Let $(X_0, X_1)$ be a compatible couple of quasi-Banach spaces. Then for every $f\in X_0+X_1$ the map $t\mapsto K(f, t, X_0, X_1)$ is nonnegative, nondecreasing and concave on $(0, \infty)$. Consequently,
\begin{equation}\label{eq:KFunctionalIntegral}
    K(f, t, X_0, X_1)=K(f, 0+, X_0, X_1)+\int_0^t k(f, s, X_0, X_1)\d s,
\end{equation} where $t\mapsto k(f, t, X_0, X_1)$ is the uniquely determined nonincreasing and right-continuous function.
\end{thm}
Following characterization states when the first term of the righthand side of $\eqref{eq:KFunctionalIntegral}$ can be omitted.  Note that since the spaces involved need not be normed, one should be familiar with a \emph{generalised Riesz–Fischer theorem} \cite[Theorem 3.3]{NePe}.
\begin{thm}\label{thm:KFunctionalDensity}
    Let $(X_0, X_1)$ be a compatible couple of quasi-Banach spaces. Then
    $$K(f, 0+, X_0, X_1)=0,\quad f\in X_0+X_1$$ if and only if $X_0\cap X_1$ is dense in $X_0$.
\end{thm}

\begin{thm}\label{thm:KInequality}
Let $(X_0, X_1)$ and $(Y_0, Y_1)$ be two compatible couples of quasi-Banach spaces. Let $T$ be a sublinear operator such that $$T\colon X_0\to Y_0\quad\mbox{and}\quad T\colon X_1\to Y_1.$$ Then there is $c>0$ such that 
\begin{equation*}\label{eq:KInequality}
K(Tf, t, Y_0, Y_1)\lesssim K(f, ct, X_0, X_1),\quad f\in X_0+X_1, t>0.
\end{equation*}
\end{thm}

In many theorems, we will use a certain elementary decomposition of $f$, to which we will refer as the \emph{optimal decomposition}.
\begin{defn}[optimal decomposition]
Let $f\in\mathcal M(0, 1)$ and $t\in (0, 1)$ be given. We define the optimal decomposition of $f$ at point $t$ by
$$f_0(s)=\min\{\abs{f(s)}, f^*(t)\}\sgn f(s),$$
and 
$$f_1(s)=\max\{\abs{f(s)}-f^*(t), 0\}\sgn f(s).$$ 
Then $f=f_0+f_1$ and it further satisfies
\begin{equation}\label{E:f_0-f_1}
    \begin{split}
        &f_0^*(s)=\min\{f^*(s), f^*(t)\},
            \\
        &f^*_1(s)=(f^*(s)-f^*(t))\chi_{(0,t)}(s),
    \end{split}
\end{equation}
and $f^*=f_0^*+f_1^*$.
\end{defn}

\begin{defn}
Let $X_0$, $X_1$ and $X$ be quasi-Banach spaces which all embed to a Hausdorff topological vector space $\mathcal H$ and satisfy $X_0\subset X \subset X_1$. We say that X is an \emph{interpolation space} between $X_0$ and $X_1$, the fact being denoted $X\in \INT(X_0, X_1)$, if for any linear operator $T$ the following holds: $$T\colon X_0\to X_0\quad \mbox{and} \quad T\colon X_1\to X_1\quad \implies\quad T\colon X\to X.$$
\end{defn}

The next theorem appears to be indispensable in the proof of Theorem~\ref{thm:ZOptimal} and follows from \cite[Chapter 5, Theorem 1.19]{BS}.
\begin{thm}\label{thm:SuperInterpolation}
    Let $(X_0, X_1)$ and $(Y_0, Y_1)$ be two compatible couples of quasi-Banach spaces and $\lambda$ be an r.i.~norm. Suppose $X_0\cap X_1$ is dense in $X_0$ and that $Y_0\cap Y_1$ is dense in $Y_0$. Set $\alpha(f)=\lambda(k(f, t, X_0, X_1))$ and $\beta(f)=\lambda(k(f, t, Y_0, Y_1))$ for $f\in\mathcal M_+(0, 1)$. Then for any linear operator $T$ satisfying $$T\colon X_0\to Y_0\quad\mbox{and}\quad T\colon X_1\to Y_1,$$ we have $$\beta(Tf)\lesssim \alpha(f),\quad f\in \mathcal M_+(0, 1).$$
\end{thm}
\begin{rem}
    It is important that the functional $\lambda$ in the theorem above is an r.i.~norm, so that we have the HLP principle at our disposal.
\end{rem}


\section{Operators involving suprema}\label{supr}

In this section we introduce two operators involving suprema, $S_I$ and $T_I$, and explore their boundedness and interpolation properties. These  operators will be defined in terms of a nondecreasing function $I\colon (0, 1)\to (0, 1)$. Most theorems concerning operator $S_I$ in this section will work in a very general setting, essentially requiring only $\Delta_2$ condition. This condition guarantees us that the spaces we will work with are at least quasinormed, as the following proposition suggests.

\begin{prop}\label{prop:fact}
Let $I\colon (0, 1)\to (0, 1)$ be a nondecreasing function. Then $m_I$ is an~r.i.q. space if and only if $I\in\Delta_2$.
\end{prop}
\begin{proof}
“$\Rightarrow$” Let $t<\frac{1}{2}$ and put $f=\chi_{(0, t)}$ and $g=\chi_{(t, 2t)}$. Then $g^*=f^*=f$ and $(f+g)^*=\chi_{(0, 2t)}$. As $\|\cdot\|_{m_I}$ is a quasinorm, we have
$$I(2t)=\|f+g\|_{m_I}\lesssim \|f\|_{m_I}+\|g\|_{m_I}\approx I(t).$$

$\Leftarrow$ We check axioms of r.i.q.~norms. From axiom (Q1), only the triangle inequality requires some comment. To this end, let $f, g \in M_+(0, 1)$ be given. Then
\begin{align*}
\|f+g\|_{m_I}\leq \sup_{t\in (0, 1)}I(t)\left(f^*\left(\frac{t}{2}\right)+g^*\left(\frac{t}
  {2}\right)\right)\lesssim \sup_{t\in (0, 1)}I\left(\frac{t}{2}\right)f^*\left(\frac{t}{2}\right)+\sup_{t\in (0, 1)}I\left(\frac{t}{2}\right)g^*\left(\frac{t}{2}\right)\leq \|f\|_{m_I}+\|g\|_{m_I}.
\end{align*}
If $0\leq f\leq g$, then $f^*\leq g^*$ and so $\|f\|_{m_I}\leq \|g\|_{m_I}$.
The fact that $I$ is bounded implies $\chi_{(0, 1)}\in m_I$. 
Finally, if $0\leq f_n\nearrow f$, then $f_n^*\nearrow f^*$ and so $\|f_n\|_{m_I}\nearrow \|f\|_{m_I}$.
\end{proof}

Regarding the operator $T_I$, we will consider an average-type condition $\eqref{eq:00001}$, which, combined with quasiconcavity of $I$, in fact implies boundedness of said operator on $L^1$. Moreover, the assumption of quasiconcavity allows us to work around the supremum appearing in its definition.

For the remainder of the paper, whenever we mention a quasiconcave function $I$, we implicitly assume that $I\colon (0, 1)\to (0, 1)$ is a bijection with $I(0+)=0$ and $I(1-)=1$. Note that this implies that $I(t)\geq t, t\in (0, 1)$.

We proceed by defining two supremum operators.

\begin{defn}\label{def:SIaTI}
Let $I$ be a nondecreasing function. We define supremum operators $S_I$ and $T_I$ on $\mathcal M_+(0, 1)$ by 
$$(S_If)(t)\coloneqq \frac{1}{I(t)}\sup_{0<s\leq t}I(s)f^*(s),\quad f\in\mathcal M_+(0, 1), t\in (0, 1),$$
and
$$(T_If)(t)\coloneqq \frac{I(t)}{t}\sup_{t\leq s<1}\frac{s}{I(s)}f^*(s),\quad f\in\mathcal M_+(0, 1), t\in (0, 1).$$
\end{defn}

Observe that $f^*\leq T_If$ and $f^*\leq S_If$ for every $f\in\mathcal M_+(0, 1)$. We also see that both of these operators are monotone -- if $f, g\in\mathcal M_+(0, 1)$ are such that $f\leq g$, then $S_I f\leq S_I g$ and $T_I f\leq T_Ig$. 

If we assume that a function $I$ is quasiconcave, $t\mapsto T_If(t)$ is a nonincreasing function for every $f\in\mathcal M_+(0, 1)$.

Using~$\eqref{eq:Almostsubadditivity}$ and assuming $I$ to satisfy $\Delta_2$ condition, we see that 
\begin{equation}\label{eq:SIsubadditivity}
(S_I(f+g))(t)\lesssim (S_If)\left(\frac{t}{2}\right)+(S_Ig)\left(\frac{t}{2}\right),\quad f\in\mathcal M_+(0, 1), t\in (0, 1),
\end{equation}
and
\begin{equation*}
(T_I(f+g))(t)\lesssim (T_If)\left(\frac{t}{2}\right)+(T_Ig)\left(\frac{t}{2}\right),\quad f\in\mathcal M_+(0, 1), t\in (0, 1).
\end{equation*}

Before we begin exploring the mapping properties of the operator $S_I$, let us observe that $t\mapsto S_If (t)$ is a nonincreasing function for every $f\in\mathcal M_+(0, 1)$.

\begin{lemma}\label{lem:SMonotone}
Let $I$ be a nondecreasing function and $f\in\mathcal M_+(0, 1)$. Then $S_I f$ is a nonincreasing function on $(0, 1)$.
\end{lemma}
\begin{proof}
Put $Rf(t)=\sup_{0<s\leq t}I(s)f^*(s)$ for $t\in (0, 1)$. Let $0<t_1<t_2<1$ be given. We consider two cases: If $Rf(t_1)=Rf(t_2)$, then $(S_I f)(t_2)\leq (S_I f)(t_1)$ because $t\mapsto \frac{1}{I(t)}$ is nonincreasing. 
If $Rf(t_1)<Rf(t_2)$, we consider a function 
$$f_1(s)\coloneqq\begin{cases}f^*(s),& s\leq t_1,\\ f^*(t_1),& t_1< s<1.\end{cases}$$ Then $f_1^*=f_1$ and $f^*\leq f_1^*$. Hence, as $Rf(t_1)< Rf(t_2)$, we have $$Rf(t_2)=\sup_{t_1<s\leq t_2} I(s)f^*(s)$$ and, consequently, $$Rf_1(t_2)=\sup_{t_1<s\leq t_2}I(s)f_1^*(s).$$ We estimate
\begin{align*}
(S_If)(t_2)&\leq (S_I f_1)(t_2)=\frac{1}{I(t_2)}\sup_{t_1<s\leq t_2} I(s)f_1^*(s)\\
	&=\frac{1}{I(t_2)}\cdot I(t_2)f_1^*(t_2)=f_1^*(t_1)=f^*(t_1)\leq (S_I f)(t_1).
\end{align*}
\end{proof}

\begin{thm}\label{thm:bSd2}
Let $I$ be a nondecreasing function. Then the operator $S_I$ has the following endpoint mapping properties:
\begin{enumerate}[(i)]
	\item $S_I\colon L^\infty\to L^\infty$,
	\item $S_I\colon m_I\to m_I$.
\end{enumerate}
\end{thm}
\begin{proof}
(i)  Given $f\in L^\infty$ we estimate
\begin{align*}
\|S_I f\|_\infty=\sup_{0<t< 1}\frac{1}{I(t)}\sup_{0<s\leq t}I(s)f^*(s)\leq \|f\|_\infty \sup_{0<t< 1}\frac{1}{I(t)}\sup_{0<s\leq t}I(s)= \|f\|_\infty \sup_{0<t< 1}\frac{1}{I(t)}\cdot I(t)= \|f\|_\infty.
\end{align*}

(ii) Let $f\in m_I$ be given. By Lemma~\ref{lem:SMonotone} we have
\begin{align*}
\|S_I f\|_{m_I}=\sup_{0<t<1}I(t)\left(\tau\mapsto \frac{1}{I(\tau)}\sup_{0<s\leq \tau}I(s)f^*(s)\right)^*(t)= \sup_{0<t< 1}I(t)\cdot\frac{1}{I(t)}\sup_{0<s\leq t}I(s)f^*(s)=\|f\|_{m_I}.
\end{align*}
\end{proof}

\begin{thm}\label{thm:BoundednessOfTIm}
Let $I$ be a quasiconcave function. Then the following holds:
\begin{enumerate}[(i)]
	\item $T_I\colon m_{\widetilde I}\to m_{\widetilde I}$,
	\item $T_I\colon L^1\to L^1$ if and only if $\int_0^t \frac{I(s)}{s}\d s\lesssim I(t)$ for $t\in (0, 1)$.
\end{enumerate}
\end{thm}
\begin{proof}
(i) The proof is similar to the proof of (ii) of the previous theorem.

(ii) Follows from \cite[Theorem 3.2]{GoOpPi} with weights $u(t)=\frac{t}{I(t)}$, $w=\frac{1}{u}$ and $v=1$ in their notation.
\end{proof}

\begin{prop}\label{lem:NIApproxMNI}
Let $I$ be a nondecreasing function. Then $I$ satisfies the average property if and only if 
\begin{equation}\label{eq:SlabyASilnyMarc} \sup_{0<s\leq t}I(s)f^{**}(s)\approx \sup_{0<s\leq t}I(s)f^*(s), \quad f\in \mathcal M_+(0, 1),  \ t\in (0, 1).\end{equation}

In this case, we have $m_I=M_I$ with equivalent norms, where $M_I$ is the Marcinkiewicz space with norm given by $$\|f\|_{M_I}=\sup_{0<t<1}I(t)f^{**}(t).$$
\end{prop}
\begin{proof}
“$\Rightarrow$” Fix $f\in\mathcal M_+(0, 1)$ and $t\in (0, 1)$. Denoting $$M=\sup_{0<s\leq t}I(s)f^*(s),$$
which we can without loss of generality assume to be finite, we have for every $s\in (0, t)$ that
$$f^*(s)\leq M\frac{1}{I(s)}.$$
Thus
\begin{align*}
\sup_{0<s\leq t}I(s)f^{**}(s)&\leq M \sup_{0<s\leq t}\frac{I(s)}{s}\int_0^s \frac{{\rm d} r}{I(r)}\lesssim M,
\end{align*}where the second inequality is exactly the average property. The converse inequality holds trivially, as $f^*\leq f^{**}$.

“$\Leftarrow$” Follows from $\eqref{eq:SlabyASilnyMarc}$ by choosing $f=\frac{1}{I}$.
\end{proof}

\begin{rem} One might ask whether it is truly necessary to work with the space $m_I$, whether the classical $M_I$ space would do the job. It can be shown that the boundedness of $S_I$ on $M_I$ enforces the average property of $I$; it suffices to test the inequality by $f=\chi_{(0, r)}, r\in (0, 1)$. Consequently, if one wants to work with more general functions in the fashion we do, the quasi-Banach spaces cannot be avoided.\end{rem}

It is evident that the operator $S_I$ and the space $m_I$ are intertwined in a sense that $S_I f$ is finite if and only if $f\in m_I$. Next theorem provides us with a result which essentially says that $S_I$ is the greatest operator which is bounded simultaneously on $L^\infty$ and $m_I$. To prove this, we require a $K$-functional related lemma first. The proof is carried out in the spirit of \cite{MM}, but we need a slightly stronger result, and, moreover, we work on a space of finite measure. We note that a result in a similar manner appeared recently in~\cite[Proposition~2.3]{Tak:23}

\begin{lemma}\label{lem:KFunctional}
Let $I\in\Delta_2$ be an increasing bijection on $(0, 1)$. Then $$K(f, t, m_I, L^\infty)\approx \sup_{0<s\leq I^{-1}(t)}I(s)f^*(s),\quad f\in m_I, t\in (0, 1).$$
\end{lemma}
\begin{proof}
Let $f\in m_I$ and $t\in (0, 1)$ be given. Write $f=f_0+f_1$, where $f_0\in m_I$ and $f_1\in L^\infty$. By Proposition~\ref{prop:fact} $m_I$ is an r.i.q.~space and so, using the boundedness of the dilation operator, the monotonicity of $\|\cdot\|_{m_I}$ and $\eqref{eq:Almostsubadditivity}$, we estimate
\begin{align*}
\sup_{0<s\leq I^{-1}(t)}I(s)f^*(s)&=\|f^*\chi_{(0, I^{-1}(t))}\|_{m_I}\lesssim\|f^*(2s)\chi_{(0, I^{-1}(t))}(2s)\|_{m_I}\\
	&\leq\|f^*(2s)\chi_{(0, I^{-1}(t))}(s)\|_{m_I}\lesssim \left\|f_0^*\chi_{(0, {\varphi_X^{-1}(t)})}\right\|_{m_I}+\left\|f_1^*\chi_{(0, {\varphi_X^{-1}(t)})}\right\|_{m_I}\\
	&\leq \|f_0\|_{m_I}+\|f_1\|_\infty\cdot\|\chi_{(0, I^{-1}(t))}\|_{m_I}=\|f_0\|_{m_I}+t\|f_1\|_\infty.
\end{align*}
 On taking infimum over all such decompositions we obtain $$\sup_{0<s\leq I^{-1}}I(s)f^*(s)\lesssim K(f, t, m_I, L^\infty).$$

To prove the converse inequality, let $f_0$ and $f_1$ be the optimal decomposition of $f$ at point $I^{-1}(t)$ from $\eqref{E:f_0-f_1}$. Then
\begin{align*}
K(f, t, m_I, L^\infty)&\leq \|f_1\|_{m_I}+t\|f_0\|_\infty=\sup_{0<s\leq I^{-1}(t)}I(s)f_1^*(s)+tf^*(I^{-1}(t))\\
&\leq \sup_{0<s\leq I^{-1}(t)}I(s)f^*(s)+tf^*(I^{-1}(t))\leq 2 \sup_{0<s\leq I^{-1}(t)}I(s)f^*(s).
\end{align*}
\end{proof}

\begin{thm}\label{thm:SImIsNice}
Let $I\in\Delta_2$ be an increasing bijection on $(0, 1)$ and let $S$ be a sublinear operator defined on $m_I$. If
$S$ is bounded on $L^\infty$ and on $m_I$, then
\begin{equation*}
(Sf)^*(t)\lesssim (S_I f)(t),\quad f\in\mathcal M_+(0, 1), t\in (0, 1).
\end{equation*}

If, in addition, $I$ has the average property, then 
\begin{equation}\label{eq:2.151}
(Sf)^{**}(t)\lesssim (S_I f)(t),\quad f\in\mathcal M_+(0, 1), t\in (0, 1),
\end{equation}
and so 
\begin{equation}\label{eq:2.152}
(S_I f)^{**}(t)\lesssim (S_I f)(t),\quad f\in\mathcal M_+(0, 1), t\in(0, 1).
\end{equation}

In particular, for an r.i.~space $X\subset m_I$ we have $X\in\INT(L^\infty, m_I)$ whenever $S_I$ is bounded on $X$.
\end{thm}
\begin{proof}
By Lemma~\ref{lem:KFunctional}, we have $$K(f, t, m_I, L^\infty)\approx \sup_{0<s\leq I^{-1}(t)}I(s)f^*(s).$$ Fix $f\in m_I$ and $t\in (0, d)$, where $d=\min\{\frac{1}{c}, cI\left(\frac{1}{c}\right)\}$ and $c\geq1$ is a constant from Theorem~\ref{thm:KInequality}. We estimate
\begin{equation}\label{eq:2.153}
\sup_{0<s\leq I^{-1}(t)}I(s)(Sf)^*(s)\lesssim \sup_{0<s\leq I^{-1}(ct)}I(s)f^*(s).
\end{equation} Passing from $t$ to $I(t)$ we arrive at
$$\sup_{0<s\leq t}I(s)(Sf)^*(s)\lesssim \sup_{0<s\leq I^{-1}(cI(t))}I(s)f^*(s).$$
Thus, choosing $s=t$ on the left-hand side and dividing by $I(t)$ gives us
$$(Sf)^*(t)\lesssim \frac{1}{I(t)}\sup_{0<s\leq I^{-1}(cI(t))}I(s)f^*(s)\approx (S_I f)(I^{-1}(cI(t)))\leq (S_I f)(t),$$ where the last inequality stems from Lemma~\ref{lem:SMonotone} and $c$ being no less than one.

Next, assuming $S_I$ is bounded on $X$ we estimate using the boundedness of the dilation operator:
$$\|(Sf)^*\|_X\approx\|(Sf)^*\chi_{(0, d)}\|_X\lesssim \|S_If\chi_{(0, d)}\|_X\approx \|S_If\|_X\lesssim\|f\|_X.$$

Now, regarding the “in addition" part of the theorem, we only need to show $\eqref{eq:2.151}$ and $\eqref{eq:2.152}$, as the rest follows from the previous part. First, $\eqref{eq:2.151}$ is a direct application of Proposition~\ref{lem:NIApproxMNI} on the left-hand side of $\eqref{eq:2.153}$. Second, $\eqref{eq:2.152}$ holds true, because of $\eqref{eq:SIsubadditivity}$ and the boundedness of the dilation operator on r.i.q.~spaces.
\end{proof}

\begin{rem}
    Should we assume that $I$ satisfies the average property, then $X\in\INT(L^\infty, M_I)$ if and only if $ S_I\colon X\to X$. The “if” part is exactly Theorem~\ref{thm:SImIsNice}. The “only if” part is proved in \cite[Theorem 1]{CwNi}.
\end{rem}

\begin{thm}\label{cor:ZRIQ}
Let $I\in\Delta_2$ be an increasing bijection on $(0, 1)$ and let $X$ be an r.i.~space. Define $\|f\|_Z\coloneqq \|S_I f\|_X$ for $f\in\mathcal M(0, 1)$. Then $\|\cdot\|_Z$ is an r.i.q.~norm and, denoting by $Z$ the r.i.q.~space corresponding to $\|\cdot\|_Z$, we have that
\begin{equation*}
S_I\colon Z\to Z.
\end{equation*} If $I$ satisfies the average property, then $\|\cdot\|_Z$ is equivalent to an r.i.~norm, and so $Z$ is an r.i.~space.
\end{thm}
\begin{proof} 

The only property of (Q1) that requires some comment is the quasi-triangle inequality. To this end, let $f, g\in\mathcal M(0, 1)$ be given. Using $\eqref{eq:SIsubadditivity}$ and the boundedness of the dilation operator on r.i.~spaces we calculate
\begin{align*}
\|f+g\|_Z&=\|S_I(f+g)\|_X\lesssim\left\|S_If\left(\frac{t}{2}\right)+S_Ig\left(\frac{t}{2}\right)\right\|_X\\
&\leq \left\|S_If\left(\frac{t}{2}\right)\right\|_X+\left\|S_Ig\left(\frac{t}{2}\right)\right\|_X\approx \|S_I f\|_X+\|S_Ig\|_X=\|f\|_Z+\|g\|_Z.
\end{align*}
Property (P2) obviously holds. 

Let $f_n, f\in\mathcal M_+(0, 1)$ be such that $f_n\nearrow f$ a.e. Then $f_n^{*}\nearrow f^{*}$. Fix $t\in (0, 1)$ and let $K<\sup_{0<s\leq t}I(s)f^{*}(s)$. We find $t_0\in (0, t]$ such that $I(t_0)f^{*}(t_0)>K$. Then $$I(t_0)f_n^{*}(t_0)\nearrow I(t_0)f^{*}(t_0)$$ and so $(S_I f_n^{*})(t)\nearrow (S_I f^{*})(t)$. As $X$ is an r.i.~norm, we get that $\|f_n\|_Z\nearrow \|f\|_Z$ and (P3) holds.

Regarding (P4), we have $\|\chi_{(0, 1)}\|_Z=\|S_I \chi_{(0, 1)}\|_X=\|\chi_{(0, 1)}\|_X<\infty$.

Thus, functional $\|\cdot\|_Z$ is an r.i.q.~norm.  

Since $\|\cdot\|_X$ possesses property (P5), we can say the same about $\|\cdot\|_Z$, because $\|\cdot\|_X\leq\|\cdot\|_Z$.

Next, for every $f\in \mathcal M(0, 1)$, we estimate
\begin{align*}
\|S_If\|_Z&=\|S_I(S_If)\|_X=\left\|\frac{1}{I(t)}\sup_{0<s\leq t}I(s)\left(\tau\mapsto\frac{1}{I(\tau)}\sup_{0<r\leq \tau}I(r)f^*(r)\right)^*(s)\right\|_X\\
    &\leq \left\|\frac{1}{I(t)}\sup_{0<s\leq t}I(s)\cdot\frac{1}{I(s)}\sup_{0<r\leq s}I(r)f^*(r)\right\|_X=\|S_If\|_X=\|f\|_Z,
\end{align*} where the inequality comes from Lemma~\ref{lem:SMonotone}.

If $I$ satisfies the average property, Proposition~\ref{lem:NIApproxMNI} tells us that 
$$\|f\|_Z=\|S_I f\|_X\approx \|S_I f^{**}\|_X\eqqcolon \|f\|,\quad f\in\mathcal M_+(0, 1).$$

It follows from the previous and subadditivity of both $f\mapsto f^{**}$ and supremum, that $\|\cdot\|$ is an r.i.~norm and $\|\cdot\|_Z\approx \|\cdot\|$.
\end{proof}

Let us point out that the space $Z$ defined in such a manner is the largest r.i.q.~space contained in the space $X$ which admits boundedness of $S_I$.

\begin{thm}\label{lem:TIHLP}
Let $I$ be a quasiconcave function satisfying $\eqref{eq:00001}$. Then
\begin{equation}\label{eq:TIHLP}
(T_I f)^{**}(t)\lesssim (T_I f^{**})(t),\quad f\in\mathcal M_+(0, 1), t\in (0, 1).
\end{equation}
\end{thm}
\begin{proof}
Fix $f\in \mathcal M_+(0, 1)$ and $t\in (0, 1)$. Let $f_0$ and $f_1$ be the optimal decomposition $\eqref{E:f_0-f_1}$ of $f$ at point $t$.

Then, as $T_I f$ is nonincreasing and supremum is subadditive, we have
\begin{align*}
(T_I f)^{**}(t)&=\frac{1}{t}\int_0^t\frac{I(s)}{s}\sup_{s\leq r<1}\frac{r}{I(r)}f^*(r) \d s\\
	&\leq \underbrace{\frac{1}{t}\int_0^t\frac{I(s)}{s}\sup_{s\leq r<1}\frac{r}{I(r)}f_0^*(r) \d s}_{I}+\underbrace{\frac{1}{t}\int_0^t\frac{I(s)}{s}\sup_{s\leq r<1}\frac{r}{I(r)}f_1^*(r) \d s}_{II}.
\end{align*}
As for $I$, we estimate
\begin{align*}
I&=\frac{1}{t}\int_0^t\frac{I(s)}{s}\sup_{s\leq r<1}\frac{r}{I(r)}\min\{f^*(r), f^*(t)\} \d s\\
	&=\frac{1}{t}\int_0^t\frac{I(s)}{s}\max\left\{\sup_{s\leq r<t }\frac{r}{I(r)}f^*(t), \sup_{t\leq r<1 }\frac{r}{I(r)}f^*(r)\right\} \d s\\
	&=\frac{1}{t}\int_0^t\frac{I(s)}{s}\sup_{t\leq r<1}\frac{r}{I(r)}f^*(r)\d s=\frac{I(t)}{t}\cdot\frac{1}{I(t)}\int_0^t\frac{I(s)}{s}\sup_{t\leq r<1}\frac{r}{I(r)}f^*(r)\d s\\
	&\lesssim\frac{I(t)}{t}\sup_{t\leq r<1}\frac{r}{I(r)}f^*(r)=(T_I f)(t).
\end{align*}

Next, using Theorem~\ref{thm:BoundednessOfTIm}, \emph{(ii)}, we estimate
\begin{align*}
II&=\frac{1}{t}\int_0^t\frac{I(s)}{s}\sup_{s\leq r<1}\frac{r}{I(r)}f_1^*(r) \d s
	\leq \frac{1}{t}\int_0^1\frac{I(s)}{s}\sup_{s\leq r<1}\frac{r}{I(r)}f_1^*(r) \d s\\
	&\lesssim \frac{1}{t}\int_0^1 f_1^*(s)\d s=\frac{1}{t}\int_0^t (f^*(s)-f^*(t))\d s\leq f^{**}(t).
\end{align*}
Putting everything together, we arrive at what we wanted:
$$(T_I f)^{**}(t)\lesssim (T_I f(t)+f^{**}(t))\lesssim (T_I f^{**})(t).$$
\end{proof}

\begin{rem}
    Notice that in the proof we used both $\eqref{eq:00001}$ and the boundedness of $T_I$ on $L^1$, which we know are equivalent. Now, taking the limit $t\to 1^-$ in $\eqref{eq:TIHLP}$, we get
    $$\int_0^1 (T_I f)(t)\d t\lesssim \lim_{t\to 1^-}\frac{I(t)}{t}\sup_{t\leq s<1}(R_I f^*)(s)=\int_0^1 f^*(t)\d t.$$ In other words, $\eqref{eq:TIHLP}$ is equivalent to the boundedness of $T_I$ on $L^1$.
\end{rem}
We shall finish this section with an important consequence of the previous theorem.

\begin{thm}\label{cor:TIMIsBoundedOnTarget}
Let $X$ be an r.i.~space and let $I$ be a quasiconcave function satisfying $\eqref{eq:00001}$. Then 
$$\left\|\frac{t}{I(t)}(T_I f)^{**}(t)\right\|_X\lesssim \left\|\frac{t}{I(t)} f^{**}(t)\right\|_X,\quad f\in \mathcal M_+(0, 1).$$ In other words, operator $T_I$ is bounded on $Y_X'$ whenever $X$ is an r.i.~space.
\end{thm}
\begin{proof}
For $f\in\mathcal M_+(0, 1)$ we estimate, owing to Theorem~\ref{lem:TIHLP},
\begin{align*}
\left\|\frac{t}{I(t)}(T_I f)^{**}(t)\right\|_X\lesssim \left\|\frac{t}{I(t)}(T_I f^{**})(t)\right\|_X=\left\| \sup_{t\leq s<1}\frac{s}{I(s)}f^{**}(s)\right\|_X\approx \left\|\frac{t}{I(t)}f^{**}(t)\right\|_X,
\end{align*} where the “$\approx$” is $\eqref{eq:RJeGvX}$. 
\end{proof}


\section{Optimality of function spaces}\label{ch:OS}
What follows is the main section of the paper. First two subsections are devoted to expressing the norm of the optimal target space in a more explicit manner. At the end of Subsection~\ref{ss465} we present the proof of Theorem~\ref{thm:Norm1}. Finally, the third subsection deals with a simplification of the optimal domain norm and contains the proof of Theorem~\ref{thm:main}.

\subsection{Optimal target space}\label{ss465}
 We now explore the norm of the optimal target space in a general setting, requiring only the condition $\eqref{eq:00001}$.

As we said in the preliminary section of the paper, given an r.i.~space $X$, its optimal target space $Y_X$ under the map $H_I$ (in our setting) always exists. Moreover, the following inclusions hold:
\begin{equation}\label{eq:OptimalInclusion}
    Y_{L^\infty}\subset Y_X\subset Y_{L^1}.
\end{equation} The situation with optimal domain spaces is a bit different. The following statement is a specification of a result in~\cite[Proposition 3.3]{Mi} to the situation suitable for our purposes.
\begin{prop}\label{prop:OptimalDomain}
    Let $I$ be a quasiconcave function and let $Y$ be an r.i.~space. Then the functional
    $$\|f\|\coloneqq\sup_{h\sim f}\|H_I h\|_Y,\quad f\in\mathcal M_+(0, 1),$$
 is an r.i.~norm if and only if $H_I 1\in Y$. In that case $\|\cdot\|_{X_Y}\coloneqq \|\cdot\|$ is the optimal domain norm for $Y$ under the mapping $H_I$.\end{prop}

From this proposition it follows that if $\frac{1}{I}$ is integrable, then $X_{Y}$ exists for every r.i.~space $Y$ and, moreover, we have the inclusions
\begin{equation*}\label{eq:OptimalInclusion2}
    X_{L^\infty}\subset X_Y\subset X_{L^1}.
\end{equation*}

We shall now reveal the important relation between boundedness of $T_I$ and interpolation property of a function space.
\begin{prop}\label{lem:INTTargetSpace}
Let $I$ be a quasiconcave function, $X\subset \Lambda_I$ be an r.i.~space and assume that $T_I\colon X'\to X'$. Then $X\in\INT(L^\infty, \Lambda_I)$. If $I$ in addition satisfies $\eqref{eq:00001}$, then $Y_{X'}\in \INT(L^\infty, \Lambda_I)$.
\end{prop}
\begin{proof}
Let $f\in\mathcal M_+(0, 1)$ and $t\in (0, 1)$ be given. Let $f_0$ and $f_1$ be the optimal decomposition $\eqref{E:f_0-f_1}$ of $f$ at point $t$. Let $T$ be a linear operator bounded on both $L^\infty$ and $\Lambda_I$. Using the subadditivity of $f\mapsto f^{**}$ and Hardy's lemma $\eqref{eq:HardyLemma}$, we estimate
\begin{equation*}
\begin{split}
\int_0^t \frac{I(s)}{s} (Tf)^*(s)\d s&=\int_0^t \frac{I(s)}{s} (T(f_0+f_1))^*(s)\d s\leq \int_0^t \frac{I(s)}{s} \left( (Tf_0)^*(s)+(Tf_1)^*(s) \right)\d s\\
	&=\int_0^t \frac{I(s)}{s} (Tf_0)^*(s)\d s+\int_0^t \frac{I(s)}{s} (Tf_1)^*(s)\d s\eqqcolon I+ II.
\end{split}
\end{equation*}
Now, for $I$, we estimate 
\begin{align*}
\int_0^t \frac{I(s)}{s} (Tf_0)^*(s)\d s \leq \int_0^t \frac{I(s)}{s} \|Tf_0\|_\infty\d s\lesssim \int_0^t \frac{I(s)}{s} \|f_0\|_\infty\d s=\int_0^t \frac{I(s)}{s} f_0^*(s)\d s.
\end{align*}
As for the $II$, we compute
\begin{align*}
\int_0^t \frac{I(s)}{s} (Tf_1)^*(s)\d s\leq \int_0^1 \frac{I(s)}{s} (Tf_1)^*(s)\d s\lesssim \int_0^1 \frac{I(s)}{s} f_1^*(s)\d s=\int_0^t\frac{I(s)}{s} f_1^*(s)\d s.
\end{align*}
Combining the last two estimates yields $$\int_0^t \frac{I(s)}{s} (Tf)^*(s)\d s\lesssim \int_0^t \frac{I(s)}{s}f^*(s)\d s.$$ Applying Hardy's lemma $\eqref{eq:HardyLemma}$ to $h(s)=\sup_{s\leq t<1}\frac{t}{I(t)}g^*(t), g\in\mathcal M_+(0, 1), s\in (0, 1)$, we obtain
$$\int_0^1 \frac{I(s)}{s}\sup_{s\leq t<1}\frac{t}{I(t)}g^*(t)(Tf)^*(s)\d s\lesssim \int_0^1 \frac{I(s)}{s}\sup_{s\leq t<1}\frac{t}{I(t)}g^*(t) f^*(s)\d s,$$ which is nothing else than
$$\int_0^1 (T_I g)(s) (Tf)^*(s)\d s\lesssim\int_0^1 (T_I g)(s) f^*(s)\d s.$$
Finally,
\begin{align*}
\int_0^1 (Tf)^*(t)g^*(t)\d t\leq \int_0^1 (Tf)^*(t) (T_I g)(t)\d t\lesssim \int_0^1 f^*(t)(T_I g)(t)\d t\leq \|f\|_X\|T_I g\|_{X'}\lesssim \|f\|_X \|g\|_{X'}.
\end{align*} Division by $\|g\|_{X'}, g\neq 0$, followed by taking supremum over $\|g\|_{X'}\leq 1$ provides us with $\|Tf\|_X\lesssim \|f\|_X, f\in\mathcal M_+(0, 1)$.

Assume now that $I$ satisfies $\eqref{eq:00001}$. Theorem~\ref{cor:TIMIsBoundedOnTarget} guarantees that $$\|T_I f\|_{Y'_{X'}}\lesssim\|f\|_{Y'_{X'}},\quad f\in\mathcal M_+(0, 1).$$ It therefore remains to show that $Y_{X'}\subset \Lambda_I$ holds for every r.i.~space $X$. This is by $\eqref{eq:OptimalInclusion}$ equivalent to showing that $Y_{L^1}\subset \Lambda_I$. We show that, in fact, $Y_{L^1}=\Lambda_I$.
We know $$\|f\|_{Y_{L^1}'}=\|R_If^*\|_\infty=\sup_{0<t<1}\frac{t}{I(t)}f^{**}(t)=\|f\|_{M_{\widetilde I}},\quad f\in\mathcal M_+(0, 1).$$ Observe that the condition $\eqref{eq:00001}$ simply states that $\widetilde I$ satisfies the average property. Consequently, Proposition~\ref{lem:NIApproxMNI} asserts that $M_{\widetilde I}=m_{\widetilde I}$, from whence it follows that $Y_{L^1}=\Lambda_I$.
\end{proof}

Our next objective is to describe the norm of the optimal target space $Y_X$. The first step in this direction is to describe the norm of $Y_X'$ via a functional, $\|\cdot\|_Z$, such that $S_I$ is bounded thereon. We begin by exploring the mapping properties of the operator $R_I$.

\begin{lemma}\label{lem:BoundednessOfRI}
Let $I$ be a quasiconcave function satisfying $\eqref{eq:00001}$. Then the operator $R_I$ has the following mapping properties:
\begin{enumerate}[(i)]
    \item $R_I\colon m_{\widetilde I}\to L^\infty$,
    \item $R_I\colon L^1\to m_I$,
    \item $R_I\colon L^1\to (m_I)_b$.
\end{enumerate}
\end{lemma}

\begin{proof}
Let $f\in m_{\widetilde I}$ and estimate
\begin{align*}
\|R_I f\|_\infty=\sup_{0<t<1} \frac{t}{I(t)}f^{**}(t)\approx \sup_{0<t<1}\frac{t}{I(t)}f^{*}(t)=\|f\|_{m_{\widetilde I}},
\end{align*} where the approximation follows from the assumption that $\widetilde I$ satisfies the average property.

We take care of $(ii)$ and $(iii)$ in one fell swoop. Let $f_n\to f$ in $L^1$. Then, for every $t\in (0,1)$, we have $$\abs{(R_I f)(t)}\leq \|f\|_\infty \frac{1}{I(t)}\int_0^t \d s=\|f\|_\infty\frac{t}{I(t)} \lesssim \|f\|_\infty,$$ from whence it follows that if $f_n$ is a sequence of bounded functions, then $R_I f_n$ is a sequence of bounded functions. Now,  $$\abs{(R_I f_n)(t)-(R_I f)(t)}=\abs{\frac{1}{I(t)}\int_0^t (f_n(s)-f(s))\d s}\leq \frac{1}{I(t)}\|f-f_n\|_1$$
 and so,
$$\left(s\mapsto \abs{(R_I f_n)(s)-(R_I f)(s)}\right)^*(t)\leq \frac{1}{I(t)}\|f-f_n\|_1.$$
Multiplying through by $I(t)$ and taking supremum over $t\in (0, 1)$ finish the proof.
\end{proof}

\begin{prop}\label{prop:KFunctional2}
    Let $I$ be a quasiconcave function. Then
    \begin{equation*}
        K(f, t, L^1, m_{\widetilde I})\lesssim \int_0^{I^{-1}(t)} f^*(s)\d s+t\sup_{I^{-1}(t)\leq s<1} \frac{s}{I(s)}f^*(s),\quad f\in\mathcal M_+(0, 1), t\in (0, 1).
    \end{equation*}
\end{prop}
\begin{proof}
    Let $f\in\mathcal M_+(0, 1)$ and $t\in (0, 1)$ be given. Let $f_0$ and $f_1$ be the optimal decomposition $\eqref{E:f_0-f_1}$ of $f$ at point $I^{-1}(t)$ in place of $t$. Using the rearrangement invariance of both $L^1$ and $m_{\widetilde I}$, we estimate
\begin{align*}
    K(f, t, L^1, X)&\leq \|f_1\|_1+t\|f_0\|_{m_{\widetilde I}}=\|f_1^*\chi_{(0, I^{-1}(t))}\|_1+t\|f_0^*\|_{m_{\widetilde I}}\\
    &=\|f^*\chi_{(0, I^{-1}(t))}\|_1-I^{-1}(t)f^*(I^{-1}(t))+t\|f^*(I^{-1}(t))\chi_{(0, I^{-1}(t))}+f^*\chi_{(I^{-1}(t), 1)}\|_{m_{\widetilde I}}\\
    &\lesssim\|f^*\chi_{(0, I^{-1}(t))}\|_1-I^{-1}(t)f^*(I^{-1}(t))+t\varphi_{m_{\widetilde I}}(I^{-1}(t))f^{*}(I^{-1}(t))+t\|f^*\chi_{(I^{-1}(t), 1)}\|_{m_{\widetilde I}}\\
        &=\|f^*\chi_{(0, I^{-1}(t))}\|_1+t\|f^*\chi_{(I^{-1}(t), 1)}\|_{m_{\widetilde I}}\leq \int_0^{I^{-1}(t)} f^*(s)\d s+t\sup_{I^{-1}(t)\leq s<1} \frac{s}{I(s)}f^*(s).
\end{align*}
\end{proof}

\begin{thm}\label{thm:ZOptimal}
Let $I$ be a quasiconcave function satisfying $\eqref{eq:00001}$ and let $X$ be an r.i.~space. Define $\|f\|_Z=\|S_I f\|_{X'}$ for $f\in\mathcal M_+(0, 1)$. Then 
\begin{equation*}\label{eq:RenormationOfTarget}\|f\|_{Y'_X}=\left\|\frac{1}{I(t)}\int_0^t f^*(s)\d s\right\|_{X'}\approx \left\|\frac{1}{I(t)}\int_0^t f^*(s)\d s\right\|_Z,\quad f\in\mathcal M_+(0, 1)
\end{equation*} and $S_I$ is bounded on $Z$.
\end{thm}
\begin{proof} We define three functionals
\begin{align*}
\lambda(f)&\coloneqq \|f^{**}(I(t))\|_{X'},\\
\alpha(f)&\coloneqq \lambda(k(f, t, L^1, m_{\widetilde I}))\\
\mbox{ and }\\
\beta(f)&\coloneqq \lambda(k(f, t, (m_I)_b, L^\infty))
\end{align*} for $f\in\mathcal M_+(0, 1)$. We check that $\lambda$ is an r.i.~norm. The triangle inequality follows from the subadditivity of $f\mapsto f^{**}$ and the triangle inequality of $\|\cdot\|_{X'}$. The rest of (P1) obviously holds. The same goes for (P2). When $0\leq f_n\nearrow f$, then $f_n^{**}\nearrow f^{**}$ and so $f_n^{**}(I(t))\nearrow f^{**}(I(t))$ for every $t\in (0, 1)$. Hence (P3) for $\lambda$ follows from (P3) of $X'$. Regarding (P4), we note that the maximal nonincreasing rearrangement of a constant function is the function itself, and so the required property follows from its counterpart in $X'$. As for (P5), we estimate
\begin{align*}
\lambda(f)&=\|f^{**}(I(t))\|_{X'}\geq \left\|f^{**}(I(t))\chi_{(0, \frac{1}{2})}(t)\right\|_{X'}\geq \left\|f^{**}\left(I\left(\frac{1}{2}\right)\right)\chi_{(0, \frac{1}{2})}\right\|_{X'}\\
	&\approx f^{**}\left(I\left(\frac{1}{2}\right)\right)\approx f^{**}\left(\frac{1}{2}\right)\approx \|f\|_1,\quad f\in\mathcal M_+(0, 1).
\end{align*}
As $\lambda$ is obviously rearrangement invariant, depending only on the nonincreasing rearrangement of $f$, we conclude that $\lambda$ is an r.i.~norm.

In general, if we have a compatible couple of quasi-Banach spaces $(X_0, X_1)$ such that $X_0\cap X_1$ is dense in $X_0$, Theorem~\ref{thm:KFunctionalDensity} combined with the definition of $\lambda$ give us
\begin{equation}\label{eq:KFunctionalLambda}
\begin{split}
        \lambda(k(f, t, X_0, X_1))&=\left\|\frac{1}{I(t)}\int_0^{I(t)}k(f, t, X_0, X_1)\right\|_{X'}      =\left\|\frac{1}{I(t)}K(f, I(t), X_0, X_1)\right\|_{X'}.
\end{split}
\end{equation}

Since $L^1\cap m_{\widetilde I}$ is dense in $L^1$, by Proposition~\ref{prop:KFunctional2} we have
\begin{equation}\label{eq:3.2.2}
\begin{split}
    \alpha(f)&=\left\|\frac{1}{I(t)}K(f, I(t), L^1, m_{\widetilde I})\right\|_{X'}\leq \left\|\frac{1}{I(t)}\left(\int_0^t f^*(s)\d s+I(t)\sup_{t\leq s<1} \frac{s}{I(s)}f^*(s)\right)\right\|_{X'}\\
    &\leq\|f\|_{Y_X'}+\left\|\frac{t}{I(t)}(T_I f)\right\|_{X'}\lesssim \|f\|_{Y_X'},
\end{split}
\end{equation} where the last estimate comes from Theorem~\ref{cor:TIMIsBoundedOnTarget}.

Next, $(m_I)_b\cap L^\infty$ is dense in $(m_I)_b$. Using $\eqref{eq:KFunctionalLambda}$ once more we have
\begin{equation}\label{eq:3.2.3}
\begin{split}
    \beta(f)&=\left\|\frac{1}{I(t)}K(f, I(t), (m_I)_b, L^\infty)\right\|_{X'}\\
        &\geq \left\|\frac{1}{I(t)}K(f, I(t), m_I, L^\infty)\right\|_{X'}\approx\left\|\frac{1}{I(t)}\sup_{0<s\leq t}I(s)f^*(s)\right\|_{X'}=\|f\|_Z,
\end{split}
\end{equation} where the first inequality follows from enlarging the space $(m_I)_b$ to $m_I$ and hence allowing more decompositions. The approximation is an application of Lemma~\ref{lem:KFunctional}.

Adding Lemma~\ref{lem:BoundednessOfRI} to the kettle we see that all the assumptions of Theorem~\ref{thm:SuperInterpolation} are satisfied. Therefore
\begin{equation}\label{eq:3.2.1}
\beta\left(R_I f^*\right)\lesssim\alpha(f^*).
\end{equation}
Finally, chaining $\eqref{eq:3.2.3}$, $\eqref{eq:3.2.1}$ and $\eqref{eq:3.2.2}$ together, we arrive at
\begin{equation}\label{eq:alphabetaOptimal}
    \|R_I f^*\|_Z\lesssim\beta(R_I f^*)\lesssim\alpha(f^*)\lesssim \|f\|_{Y_X'}.
\end{equation}
 This finishes the proof, as the reverse inequality holds trivially.
\end{proof}

We now introduce two functionals and exhibit their equivalence.
\begin{lemma}\label{lem:ConcaveNorm}
Let $I$ be a quasiconcave function and let $X$ be an r.i. space. Define a~functional $\|f\|_Z=\|S_I f\|_X, f\in\mathcal M_+(0, 1)$. Then
\begin{equation}\label{eq:ConcaveNorm}
\mu(f)\coloneqq\sup_{\substack{\|g\|_Z\leq 1 \\\|g\|_\infty <\infty}}\int_0^1 f^*(t) \d \csup{0<s\leq t}I(s)g^*(s)+\|f\|_1,\quad f\in\mathcal M_+(0, 1),
\end{equation} is an r.i.~norm. Here, $\csup{0<s\leq t}\varphi(s)$ stands for the least concave majorant of $t\mapsto \sup_{0<s\leq t}\varphi(s)$.
\end{lemma}
\begin{proof}
Denote $h_g(t)=\csup{0<s\leq t}I(s)g^*(s)$ for $g\in\mathcal M_+(0, 1)$. Let us observe that $t\mapsto \sup_{0<s\leq t}I(s)g^*(s)$ is a quasiconcave map for every $g\in Z$ and so $h_g$ is a finite concave function. Indeed, for such functions $g$ the supremum is finite, as $Z\subset m_I$, and with increasing $t$ the supremum is nondecreasing. It thus remains to check that for $0<t_1<t_2<1$ the inequality $$\frac{\sup_{0<s\leq t_2}I(s)g^*(s)}{t_2}\leq \frac{\sup_{0<s\leq t_1}I(s)g^*(s)}{t_1}$$ holds. We calculate
\begin{align*}
    \frac{\sup_{t_1<s\leq t_2}I(s)g^*(s)}{t_2}&\leq g^*(t_1)\frac{\sup_{0<s\leq t_1}I\left(\frac{t_2}{t_1}\cdot s\right)}{\frac{t_2}{t_1}\cdot t_1}\\
    &\leq\frac{\sup_{0<s\leq t_1}I(s)g^*(t_1)}{t_1}\leq \frac{\sup_{0<s\leq t_1}I(s)g^*(s)}{t_1},
\end{align*} where in the second inequality we used the quasiconcavity of $I$. For $\|g\|_Z\leq 1$ we therefore have
$$\int_0^1 f^*(t) \d h_g(t)=\int_0^1 f^*(t) \frac{{\rm d} h_g(t)}{{\rm d} t}\d t.$$ As $t\mapsto \frac{{\rm d} h_g(t)}{{\rm d} t}$ is nonincreasing, Hardy's lemma $\eqref{eq:HardyLemma}$ gives us a triangle inequality of the functional $\mu$. Property (P2) is obvious. (P3) follows from the monotone convergence theorem. To check (P4), let $\|g\|_Z\leq 1$ be given. Using the fact that $Z\emb m_I$, we estimate
\begin{align*}
    \int_0^1 \d \csup{0<s\leq t}I(s)g^*(s)\leq \csup{0<s<1}I(s)g^*(s)\approx \sup_{0<s<1}I(s)g^*(s)=\|g\|_{m_I}\lesssim \|g\|_Z\leq 1.
\end{align*} Thus $\mu(\chi_{(0, 1)})<\infty$. Regarding (P5), $\|f\|_1\leq \mu(f), f\in\mathcal M_+(0, 1)$, follows from the definition of $\mu$. Finally, $\mu$ is rearrangement invariant, as the first expression in its definition $\eqref{eq:ConcaveNorm}$ depends only on the nonincreasing rearrangement of a function and the second term, $\|\cdot\|_1$, is rearrangement invariant.
\end{proof}

\begin{prop}\label{prop:LRIN}
     Let $I$ be a quasiconcave function and let $X$ be an r.i.~space. Put $\|f\|_Z\coloneqq \|S_I f\|_X$ for $f\in\mathcal M_+(0, 1)$. Then the functional $\lambda$ on $\mathcal M_+(0, 1)$ defined by
\begin{equation*}
    \lambda(f)=\sup_{\|g\|_Z\leq 1}\int_0^1-I(t)g^*(t)\d f^*(t)+\|f\|_1,\quad f\in\mathcal M_+(0, 1),
\end{equation*} is equivalent to an r.i.~norm.
\end{prop}
\begin{proof}
    We show that $\lambda\approx \mu$, where $\mu$ is the functional from Lemma~\ref{lem:ConcaveNorm}. First observe that, by the monotone convergence theorem, it suffices to consider only bounded functions over which we take the supremum in the definition of the functional $\lambda$. Also, without loss of generality, it suffices to consider only the functions which are finite a.e.
    
    We show their equivalence in three steps. Let first $f\in\mathcal M_+(0, 1)$ be such that $f^*(0+)<\infty$ and $f^*(1-)=0$ and pick $g\in\mathcal M_+(0, 1)$ bounded such that $\|g\|_Z\leq 1$. Then $$\lim_{t\to 0^+} f^*(t) \csup{0<s\leq t}I(s)g^*(s)=\lim_{t\to 1^-} f^*(t) \csup{0<s\leq t}I(s)g^*(s)=0.$$ Consequently, integration by parts yields 
\begin{align*}
    \int_0^1 f^*(t)\d \csup{0<s\leq t} I(s) g^*(s)=\int_0^1-\csup{0<s\leq t}I(s)g^*(s) \d f^*(t).
\end{align*} We therefore have
\begin{align*}
    &\sup_{\substack{\|g\|_Z\leq 1\\\|g\|_\infty<\infty}}\int_0^1 f^*(t)\d \csup{0<s\leq t} I(s) g^*(s)= \sup_{\substack{\|g\|_Z\leq 1\\\|g\|_\infty<\infty}}\int_0^1-\csup{0<s\leq t}I(s)g^*(s) \d f^*(t)\\
    \approx&\sup_{\substack{\|g\|_Z\leq 1\\\|g\|_\infty<\infty}}\int_0^1 -\sup_{0<s\leq t}I(s)g^*(s)\d f^*(t)=\sup_{\substack{\|g\|_Z\leq 1\\\|g\|_\infty<\infty}}\int_0^1 -I(t) (S_Ig)(t)\d f^*(t)\\
    =&\sup_{\substack{\|S_I g\|_Z\leq 1\\\|g\|_\infty<\infty}}\int_0^1 -I(t) (S_Ig)(t)\d f^*(t)=\sup_{\substack{\|g\|_Z\leq 1\\\|g\|_\infty<\infty}}\int_0^1 -I(t)g^*(t)\d f^*(t),
\end{align*} and so $\mu(f)\approx \lambda(f)$. Here we used the fact that $\|g\|_Z=\|S_Ig\|_Z$ for every $g\in\mathcal M_+(0, 1)$ and Theorem~\ref{thm:bSd2}, \emph{(i)}, which says that if $g$ is bounded, so is $S_Ig$.

Second, let $f\in\mathcal M_+(0, 1)$ satisfy $f^*(1-)=0$. Let $\{t_n\}\in (0, 1)^\Nat$ be a sequence of points of continuity of $f^*$ such that $t_n\searrow 0$. For every $n\in\Nat$ we set $f_n=f_n^*=\min\{ f^*, f^*(t_n)\}$. We now infer that $\lambda(f_n)\nearrow \lambda(f)$. Indeed, by the monotone convergence theorem we have
\begin{align*}
\lambda(f_n)&=\sup_{\|g\|_Z\leq 1}\int_0^1 -I(t)g^*(t)\d f_n^*(t)+\|f_n\|_1\\
&=\sup_{\|g\|_Z\leq 1}\int_{t_n}^1-I(t)g^*(t)\d f^*(t)+\|f_n\|_1
\nearrow \lambda(f).
\end{align*}

Finally, let $f\in\mathcal M_+(0, 1)$ be such that $f^*(1-)<\infty$. Then, using the previous step, we have
\begin{align*}
    \lambda(f)&=\sup_{\|g\|_Z\leq 1}\int_0^1 -I(t)g^*(t)\d f^*(t)+\|f\|_1\\
    &=\sup_{\|g\|_Z\leq 1}\int_0^1 -I(t)g^*(t)\d (f^*-f^*(1-))(t)+\|f^*-f^*(1-)\|_1+f^*(1-)\\
    &\approx \mu(f^*-f^*(1-))+f^*(1-)\approx\mu(f),
\end{align*} since
\begin{align*}
    \mu(f^*-f^*(1-))+f^*(1-)&\leq \mu(f)+\mu(f^*(1-))+f^*(1-)\\&\approx \mu(f)+\|f^*(1-)\|_1
        \leq \mu(f)+\|f\|_1\approx\mu(f)
\end{align*} and
$$\mu(f)\leq \mu(f^*-f^*(1-))+\mu(f^*(1-))\approx\mu(f-f^*(1-))+f^*(1-).$$
\end{proof}

\begin{rem}
Both functionals $\mu$ and $\lambda$ contain $\|\cdot\|_1$ in their definitions. For the functional $\mu$ it guarantees that $X(\mu)\emb L^1$, while for $\lambda$ it guarantees that $\lambda(f)=0\iff f=0$ a.e.
\end{rem}

Before we establish an alternative description of the optimal target space norm, we require a technical lemma which extends $\eqref{eq:RJeGvX}$ to our setting.

\begin{lemma}\label{lem:RIFGIFvZ}
    Let $I$ be a quasiconcave function and $X$ be an r.i.~space. Put $\|f\|_Z=\|S_If\|_X, f\in\mathcal M_+(0, 1)$. Then 
    \begin{equation*}
        \|R_I f^*\|_Z\approx \|G_I f\|_Z,\quad f\in\mathcal M_+(0, 1).
    \end{equation*}
\end{lemma}
\begin{proof}
    We show that $S_I G_If=G_I f$ for every $f\in\mathcal M_+(0, 1)$. To this end, fix $f\in\mathcal M_+(0, 1)$ and $t\in (0, 1)$. We calculate
\begin{align*}
    (S_I G_I f)(t)&=\frac{1}{I(t)}\sup_{0<s\leq t}I(s)\sup_{s\leq r<1}\frac{1}{I(r)}\int_0^r f^*(z)\d z\\
    &=\frac{1}{I(t)}\sup_{0<s\leq t}\sup_{s\leq r<1}I(s)\frac{1}{I(r)}\int_0^r f^*(z)\d z\\
    &=\frac{1}{I(t)}\max\left\{\sup_{0<s\leq t}\sup_{s\leq r\leq t}\frac{I(s)}{I(r)}\int_0^r f^*(z)\d z, \sup_{0<s\leq t}\sup_{t\leq r<1}\frac{I(s)}{I(r)}\int_0^r f^*(z)\d z\right\}\\
    &=\frac{1}{I(t)}\max\left\{\sup_{0<r\leq t}\sup_{0<s\leq r}\frac{I(s)}{I(r)}\int_0^r f^*(z)\d z, \sup_{t\leq r<1}\frac{I(t)}{I(r)}\int_0^r f^*(z)\d z\right\}\\
    &=\frac{1}{I(t)}\max\left\{\sup_{0<r\leq t}\int_0^r f^*(z)\d z, \sup_{t\leq r<1}\frac{I(t)}{I(r)}\int_0^r f^*(z)\d z\right\}\\
    &=\frac{1}{I(t)}\max\left\{\frac{I(t)}{I(t)}\int_0^t f^*(z)\d z, \sup_{t\leq r<1}\frac{I(t)}{I(r)}\int_0^r f^*(z)\d z\right\}\\
    &=\max\left\{\frac{1}{I(t)}\int_0^t f^*(z)\d z, \sup_{t\leq r<1}\frac{1}{I(r)}\int_0^r f^*(z)\d z\right\}\\
    &=\sup_{t\leq r<1} \frac{1}{I(r)}\int_0^r f^*(z)\d z=G_If (t).
\end{align*} Therefore, by $\eqref{eq:RJeGvX}$,
\begin{align*}
    \|G_If \|_Z=\|S_I G_I f\|_X=\|G_I f\|_X\approx \|R_I f^*\|_X\leq\|S_I R_I f^*\|_X= \|R_I f^*\|_Z.
\end{align*}
We know $R_I f^*\leq G_I f$ and so $S_I R_I f^*\leq S_I G_I f$. Furnishing this last inequality with $\|\cdot\|_X$ finishes the proof.
\end{proof}

We will now use results of the last two sections to prove Theorem~\ref{thm:Norm1}.
\begin{proof}[Proof of Theorem~\ref{thm:Norm1}]
Put $\|f\|_Z=\|S_I f\|_{X'}$ for $f\in\mathcal M_+(0, 1)$. We need to show that $\|f\|_{Y_X}\approx \lambda(f)$ for $f\in\mathcal M_+(0, 1)$, where $\lambda$ is the functional from Proposition~\ref{prop:LRIN}. Let $f,g \in\mathcal M_+(0, 1)$ be given and assume $f^*(0+)<\infty$ and $f^*(1-)=0$. Then

\begin{align*}
    \int_0^1 f^*(s)g^*(s)\d s&=\int_0^1 -g^*(s)\int_s^1 {\rm d} f^*(t)\d s=\int_0^1-\int_0^t g^*(s)\d s\d f^*(t)\\
    &=\int_0^1-\frac{I(t)}{I(t)}\int_0^t g^*(s)\d s\d f^*(t)\leq\int_0^1 -I(t)(G_I g)(t) \d f^*(t)\\
    &\approx\frac{\|g\|_{Y_X'}}{\left\|\frac{1}{I(t)}\int_0^t g^*\right\|_Z}\int_0^1-I(t)(G_I g)(t)\d f^*(t)\\
    &\approx \frac{\|g\|_{Y_X'}}{\left\|G_I g\right\|_Z}\int_0^1-I(t)(G_I g)(t)\d f^*(t)\\
    &\leq \|g\|_{Y_X'}\cdot\sup_{\|g\|_Z\leq 1}\int_0^1-I(t)g^*(t)\d f^*(t)\leq\|g\|_{Y_X'}\cdot\lambda(f),
\end{align*} where the first approximation is Theorem~\ref{thm:ZOptimal} and in the second approximation we used Lemma~\ref{lem:RIFGIFvZ}. Dividing by $\|g\|_{Y_X'}$ and taking supremum over $g\in\mathcal M_+(0, 1)$ with $\|g\|_{Y_X'}\leq 1$ gives us $\|f\|_{Y_X}\lesssim \lambda(f)$.

In the other direction, let $f\in\mathcal M_+(0, 1)$ be arbitrary. As $$\|f\|_{Y_X'}=\left\|\frac{1}{I(t)}\int_0^t f^*\right\|_{X'},$$ there exists $h\in\mathcal M_+(0, 1)$ with $\|h\|_X\leq 1$ and $$\frac{1}{2}\|f\|_{Y_X'}\leq\int_0^1 \frac{h(t)}{I(t)}\int_0^t f^*(s)\d s\d t.$$ Put $g(t)=g^*(t)=\int_t^1\frac{h(s)}{I(s)}\d s$. For $k\in\mathcal M_+(0, 1)$ satisfying $\|k\|_Z\leq 1$ we have
\begin{align*}
    \int_0^1 -I(t)k^*(t)\d g^*(t)&=\int_0^1 I(t)k^*(t) \frac{h(t)}{I(t)}\d t\\
    &=\int_0^1 k^*(t) h(t)\d t\leq \|k\|_{X'}\|h\|_X\leq \|k\|_Z\|h\|_X\leq 1.
\end{align*} By Fubini's theorem we estimate $$\int_0^1 g^*(t)\d t=\int_0^1\int_t^1 \frac{h(s)}{I(s)}\d s\d t=\int_0^1 h(s) \frac{s}{I(s)}\d s\leq\int_0^1 h(s)\d s\lesssim \|h\|_X\leq 1.$$ The last two estimates tell us that $\lambda(g)\lesssim 1$.

We further have $$\int_0^1 f^*(t)g^*(t)\d t=\int_0^1 f^*(t) \int_t^1 \frac{h(s)}{I(s)}\d s\d t=\int_0^1 \frac{h(s)}{I(s)}\int_0^s f^*(t)\d t\d s\geq\frac{1}{2}\|f\|_{Y_X'}.$$ Finally, putting everything together yields
$$\frac{1}{2}\|f\|_{Y_X'}\leq \int_0^1 f^*(t) g^*(t)\d t\leq \lambda'(f)\lambda(g)\lesssim \lambda'(f)$$ or, equivalently, $\lambda(f)\lesssim \|f\|_{Y_X}$.
\end{proof}

\subsection{Alternative norm in the optimal target  space}

The goal of this section is to express the optimal target norm using the maximal nonincreasing rearrangement. We will need to strengthen the assumptions we impose on the function $I$ and in turn obtain a helpful tool in the proof of the main theorem and a new way to describe the norm of functions. In view of Proposition~\ref{lem:INTTargetSpace} we will need boundedness of $f\mapsto f^{**}$ on $\Lambda_I$.



\begin{prop}\label{prop:Norm1}
Let $I$ be a quasiconcave function satisfying $\eqref{eq:00001}$ and $\eqref{eq:BoundedMaximalOp}$. Let $X$ be an r.i.~space. Put $\|f\|_Z=\|S_I f\|_{X'}, f\in\mathcal M_+(0, 1)$. Then 
\begin{equation*}
\|f\|_{Y_X}\approx\sup_{\|g\|_Z\leq 1}\int_0^1 \frac{I(t)}{t}(f^{**}(t)-f^*(t))g^*(t)\d t+\|f\|_1,\quad f\in\mathcal M_+(0, 1).
\end{equation*}
\end{prop}

\begin{proof}
Thanks to the fact that $I$ satisfies~\eqref{eq:BoundedMaximalOp}, we can use~\cite[Theorem 10.3.12]{FS} to establish that the mapping $f\mapsto f^{**}$ is bounded on $\Lambda_I$. This implies that the map $f\mapsto \frac{1}{t}\int_0^t f(s)\d s$ is, too, bounded on $\Lambda_I$. Proposition~\ref{lem:INTTargetSpace} tells us that $Y_X\in\INT(L^\infty, \Lambda_I)$. Therefore $f\mapsto\frac{1}{t}\int_0^t f(s)\d s$ and, consequently, $f\mapsto f^{**}$ are bounded on $Y_X$.

Now, for every $\eps\in (0, \frac{1}{2})$ and $g\in\mathcal M_+(0, 1)$ we have 
\begin{align*}
&\int_\eps^{1-\eps}-I(t)g^*(t) \d\left[\frac{1}{t}\int_0^t f^{**}(s)\d s\right]\\=&\int_\eps^{1-\eps}-I(t)g^*(t)\left[-\frac{1}{t^2}\int_0^t f^{**}(s)\d s+\frac{1}{t}f^{**}(t)\right]{\rm d} t\\
=&\int_\eps^{1-\eps}I(t)\left[\frac{1}{t^2}\int_0^t(f^{**}(s)-f^*(s))\d s\right]g^*(t)\d t.
\end{align*}
Thus, passing to the limit and using Fubini's Theorem, we have
\begin{align*}
-\int_0^1 I(t)g^*(t)\d \left[\frac{1}{t}\int_0^t f^{**}(s)\d s\right]&=\int_0^1 \frac{I(t)}{t^2}\int_0^t(f^{**}(s)-f^*(s))\d s\,g^*(t)\d t\\
&=\int_0^1 \frac{I(s)}{s}(f^{**}(s)-f^*(s))\frac{s}{I(s)}\int_s^1\frac{I(t)}{t^{2}}g^*(t)\d t\d s\\
&=\int_0^1 \frac{I(s)}{s}(f^{**}(s)-f^*(s)) (H g^*)(s).
\end{align*} Here, the operator $H$ is defined as 
\begin{equation}\label{eq:OperatorH}
    Hg(s)=\frac{s}{I(s)}\int_s^1\frac{I(t)}{t^2}g(t)\d t,\quad g\in\mathcal M_+(0, 1), s\in (0, 1).
\end{equation} Observe that, on one hand, we have 
\begin{equation}\label{eq:Hjeleqthg}
\begin{split}
(H g^*)(s)=\frac{s}{I(s)}\int_s^1\frac{I(t)}{t^{2}}g^*(t)\d t&\leq g^*(s)\frac{s}{I(s)}\int_s^1 \frac{I(t)}{t^{2}}\d t\lesssim g^*(s),
\end{split}
\end{equation}and, on the other hand,
\begin{equation}\label{eq:ansdhbsajhd}
(H g^*)\left(\frac{s}{2}\right)\gtrsim \frac{s}{I(s)}\int_\frac{s}{2}^s g^*(t)\frac{I(t)}{t^2}\d t\gtrsim \frac{s}{I(s)}g^*(s)\frac{I(s)}{s^2}\cdot s=g^*(s).
\end{equation} Finally, using Theorem~\ref{thm:Norm1}, we get
\begin{equation}\label{eq:nejakarovnice}
\begin{split}
\|f\|_{Y_X}&\approx\left\|\frac{1}{t}\int_0^t f^{**}(s)\d s\right\|_{Y_X}\\&\approx\sup_{\|g\|_Z\leq 1}\int_0^1\frac{I(t)}{t}(f^{**}(t)-f^*(t)) (Hg^*)(t)\d t+\|f\|_1.
\end{split}
\end{equation} To finish the proof we need to show that $$\sup_{\|g\|_Z\leq 1}\int_0^1\frac{I(t)}{t}(f^{**}(t)-f^*(t)) (Hg^*)(t)\d t\approx \sup_{\|g\|_Z\leq 1}\int_0^1\frac{I(t)}{t}(f^{**}(t)-f^*(t)) g^*(t)\d t.$$  The inequality “$\lesssim$” immediately follows from $\eqref{eq:Hjeleqthg}$. 

The boundedness of the dilation operator followed by $\eqref{eq:nejakarovnice}$ and the fact that $(f^*(2\cdot))^{**}(t)=f^{**}(2t)$, substitution and $\eqref{eq:ansdhbsajhd}$ yield
\begin{align*}
    \|f\|_{Y_X}\approx \|f^{*}(2t)\|_{Y_X}&\approx\sup_{\|g\|_Z\leq 1}\int_0^1\frac{I(t)}{t}(f^{**}(2t)-f^*(2t)) (Hg^*)(t)\d t+\|f\|_1\\
    &\gtrsim\sup_{\|g\|_Z\leq 1}\int_0^1\frac{I(t)}{t}(f^{**}(t)-f^*(t)) (Hg^*)\left(\frac{t}{2}\right){\rm d} t+\|f\|_1\\
    &\gtrsim \sup_{\|g\|_Z\leq 1}\int_0^1\frac{I(t)}{t}(f^{**}(t)-f^*(t)) g^*(t)\d t+\|f\|_1.
\end{align*}
\end{proof}

\begin{thm}\label{thm:Norm2}
 Let $I\in \mathcal Q$. Let $X$ be an~r.i.~space such that $S_I$ is bounded on $X'$. Then
\begin{equation}\label{eq:OTargetNorm1212}
\|f\|_{Y_X}\approx\left\|\frac{I(t)}{t}(f^{**}(t)-f^*(t))\right\|_X+\|f\|_1,\quad f\in\mathcal M_+(0, 1).
\end{equation}
\end{thm} 

\begin{proof}
Since $S_I$ is bounded on $X'$, we have $\|S_I g\|_{X'}\approx \|g\|_{X'}, g\in\mathcal M_+(0, 1)$. Therefore, by Proposition~\ref{prop:Norm1}, $\eqref{eq:LevelFunction}$ and the fact that $\|\cdot\|_X=\|\cdot\|_{X''}$, we obtain
\begin{equation}\label{eq:rovniiiceee}
\begin{split}
\|f\|_{Y_X}&\approx \sup_{\|S_I g\|_{X'}\leq1}\int_0^1 \frac{I(t)}{t}(f^{**}(t)-f^*(t))g^*(t)\d t+\|f\|_1\\
    &\approx \sup_{\|g\|_{X'}\leq1}\int_0^1 \frac{I(t)}{t}(f^{**}(t)-f^*(t))g^*(t)\d t+\|f\|_1\\
	&=\left\|\frac{I(t)}{t}(f^{**}(t)-f^*(t))\right\|_{(X')_d'}+\|f\|_1\\
	&=\left\|\left(\frac{I(s)}{s}(f^{**}(s)-f^*(s))\right)^\circ\right\|_X+\|f\|_1,\quad f\in\mathcal M_+(0, 1).
\end{split}
\end{equation}
It follows from $\eqref{eq:levelfunctionandrearrangement}$ and the HLP principle $\eqref{eq:HLP}$ that $\|f^\circ\|\leq \|f\|$ for every $f\in\mathcal M_+(0, 1)$ and every r.i.~norm $\|\cdot\|$. We have therefore obtained “$\lesssim$" in $\eqref{eq:OTargetNorm1212}$.

From the definition of the level function, for every $f\in\mathcal M_+(0, 1)$, we have $$\int_0^t \frac{I(s)}{s}(f^{**}(s)-f^*(s))\d s\leq \int_0^t\left(\frac{I(y)}{y}(f^{**}(y)-f^*(y))\right)^\circ(s)\d s,\quad t\in (0, 1).$$ Using Hardy's lemma we get to $$\int_0^t \frac{I(s)}{s}(f^{**}(s)-f^*(s))g^{**}(s)\d s\leq \int_0^t\left(\frac{I(y)}{y}(f^{**}(y)-f^*(y))\right)^\circ(s)g^{**}(s)\d s$$ for every $t\in (0, 1), f, g\in\mathcal M_+(0, 1)$. Using Fubini's theorem, we rewrite the last inequality as 
\begin{equation}\label{eq:HLPLF}
\int_0^t g^*(s)\int_s^t \frac{I(y)}{y}(f^{**}(y)-f^*(y))\frac{{\rm d} y}{y}\d s \leq \int_0^t g^*(s)\int_s^t \left(\frac{I(z)}{z}(f^{**}(z)-f^*(z))\right)^\circ(y)\frac{{\rm d} y}{y}\d s.
\end{equation}

For $f\in\mathcal M_+(0, 1)$ we calculate
\begin{equation}\label{eq:rovnice00111}
\begin{split}
&\int_t^1 \frac{I(s)}{s}(f^{**}(s)-f^*(s))\frac{{\rm d} s}{s}\\
=&\int_t^1 \frac{I(s)}{s^{3}}\int_0^1\chi_{(0, s)}(y) f^*(y)\d y\d s-\int_t^1\frac{I(s)}{s^2} f^*(s) \d s\\
=&\int_0^1 f^*(y)\int_t^1 \frac{I(s)}{s^{3}}\chi_{(y, 1)}(s)\d s\d y-\int_t^1\frac{I(s)}{s^2} f^*(s) \d s\\
=&\underbrace{\int_0^t f^*(y)\d y\int_t^1 \frac{I(s)}{s^{3}}\d s}_{I}+\underbrace{\int_t^1 f^*(y)\int_y^1\frac{I(s)}{s^{3}}\d s\d y
-\int_t^1\frac{I(s)}{s^2} f^*(s) \d s}_{II}.
\end{split}
\end{equation}

In what follows, let $c$ and $d$ be the constants from Definition~\ref{def:Conditions}. For $I$ we estimate
\begin{equation}\label{eq:rovnice000112}
\begin{split}
    I=tf^{**}(t)\int_t^1\frac{I(s)}{s^3}\d s\geq ctf^{**}(t)\left(\frac{I(t)}{t^2}-1\right)=c\frac{I(t)}{t}f^{**}(t)-ctf^{**}(t).
\end{split}
\end{equation}
As for $II$, using that $c<1$ (Remark~\ref{rem:constC}), we estimate
\begin{equation}\label{eq:486454}
\begin{split}
    II&\geq\int_t^1 f^*(y)c\left(\frac{I(y)}{y^2}-1\right){\rm d}y-\int_t^1 \frac{I(y)}{y^2}f^*(y)\d y\\
    &=\int_t^1 (c-1)f^*(y)\frac{I(y)}{y^2}\d y-c\int_t^1 f^*(y)\d y\\
&\geq (c-1)f^*(t)\int_t^1 \frac{I(y)}{y^2}\d y-c\int_t^1 f^*(y)\d y\\
&\geq (c-1)d\frac{I(t)}{t}f^*(t)-c\int_t^1 f^*(y)\d y.
\end{split}
\end{equation}
Deploying the assumption $(1-c)d\leq c$, estimates $\eqref{eq:rovnice000112}$, $\eqref{eq:486454}$ and $\eqref{eq:rovnice00111}$, we arrive at
\begin{align*}
&(1-c)d\frac{I(t)}{t}(f^{**}(t)-f^*(t))-c\|f\|_1\\
\leq& c\frac{I(t)}{t}f^{**}(t)+(c-1)d\frac{I(t)}{t}f^*(t)-c\|f\|_1\\
=& c\frac{I(t)}{t}f^{**}(t)-ctf^{**}(t)+(c-1)d\frac{I(t)}{t}f^*(t)-c\int_t^1f^*(y)\d y\\
\leq&I+II= \int_t^1 \frac{I(s)}{s}(f^{**}(s)-f^*(s))\frac{{\rm d} s}{s}.
\end{align*}
Therefore we obtain $$\frac{I(t)}{t}(f^{**}(t)-f^*(t))\lesssim \int_t^1 \frac{I(s)}{s}(f^{**}(s)-f^*(s))\frac{{\rm d} s}{s}+\|f\|_1,\quad f\in\mathcal M_+(0, 1), t\in (0, 1).$$

Hence, for every $f\in\mathcal M_+(0, 1)$ we have 
\begin{equation}\label{eq:iuahdsiuhaids}
\begin{split}
\left\|\frac{I(t)}{t}(f^{**}(t)-f^*(t))\right\|_X\lesssim \left\| \int_t^1 \frac{I(s)}{s}(f^{**}(s)-f^*(s))\frac{{\rm d} s}{s} \right\|_X+&\|f\|_1\\
	\lesssim \left\| \int_t^1 \left(\frac{I(y)}{y}(f^{**}(y)-f^*(y))\right)^\circ(s)\frac{{\rm d} s}{s} \right\|_X+&\|f\|_1,
 \end{split}
\end{equation}
where the second inequality comes from $\eqref{eq:HLPLF}$ by taking $t=1$ and supremum over all $g\in\mathcal M_+(0, 1)$ with $\|g\|_{X'}\leq 1$. 

As $I$ has the average property, by Proposition~\ref{lem:NIApproxMNI} $f\mapsto \frac{1}{t}\int_0^t f(s)\d s$ is bounded simultaneously on $L^\infty$ and $m_I$. Since the operator $S_I$ is bounded on $X'$, Theorem~\ref{thm:SImIsNice} asserts that $X'\in\INT(L^\infty, m_I)$. Therefore $f\mapsto \frac{1}{t}\int_0^t f(s)\d s$ is bounded on $X'$ and so, by duality $\eqref{eq:associateoperatorss}$, its associate operator, $f\mapsto \int_t^1 \frac{f(s)}{s}\d s$, is bounded on $X$. We therefore have $$\left\| \int_t^1 \left(\frac{I(y)}{y}(f^{**}(y)-f^*(y))\right)^\circ(s)\frac{{\rm d} s}{s} \right\|_X\lesssim \left\| \left(\frac{I(y)}{y}(f^{**}(y)-f^*(y))\right)^\circ \right\|_X.$$ Adding this estimate to $\eqref{eq:iuahdsiuhaids}$ and using $\eqref{eq:rovniiiceee}$ we have
$$\left\|\frac{I(t)}{t}(f^{**}(t)-f^*(t))\right\|_X\lesssim \left\| \left(\frac{I(y)}{y}(f^{**}(y)-f^*(y))\right)^\circ \right\|_X+\|f\|_1\approx\|f\|_{Y_X},$$ which concludes the proof.
\end{proof}

\begin{rem}\label{rem:Povidkakthm312}
    Observe that the average property was only used in Theorem~\ref{thm:SImIsNice} to guarantee us the validity of the theorem for an arbitrary r.i.~space. Therefore, if the map $f\mapsto \frac{1}{t}\int_0^t f(s)\d s$ is bounded on the associate space of a certain r.i.~space $X$, we may drop this assumption and still obtain validity of $\eqref{eq:OTargetNorm1212}$.

Also, we can view the condition about $I$ having the average property as the property of the operator $H_I$ to push the optimal target space $Y_{L^1}$ far enough from $L^1$, so that $f\mapsto f^{**}$ is bounded on $Y_{L^1}$.

Since we assume condition $\eqref{eq:00001}$, estimate $\eqref{eq:OTargetNorm1212}$ reads as  $$\|f\|_{\Lambda_I}\approx\left\|\frac{I(t)}{t}(f^{**}(t)-f^*(t))\right\|_1+\|f\|_1.$$ From here it is evident, that in order to have the right-hand side majorized by the left-hand side, the boundedness of $f\mapsto f^{**}$ on $\Lambda_I$ is necessary. This, as we remarked, is not satisfied for example by $I(t)=t\log^\alpha\frac{2}{t}, \alpha\in\left[0, \frac{1}{2}\right]$.
\end{rem}

\begin{cor}
Let $I\in \mathcal Q$ and let $X$ be an r.i.~space. Put $\|f\|_Z=\|S_I f^{**}\|_{X'}$.  Then
\begin{equation*}
\|f\|_{Y_X}\approx\left\|\frac{I(t)}{t}(f^{**}(t)-f^*(t))\right\|_{Z'}+\|f\|_1.
\end{equation*}
\end{cor}
\begin{proof}
By virtue of Theorem~\ref{cor:ZRIQ}, $Z$ is an r.i.~space which admits boundedness of the operator $S_I$. Thus, by Theorem~\ref{thm:Norm2}, we have
$$\|f\|_{Y_{Z'}}\approx\left\|\frac{I(t)}{t}(f^{**}(t)-f^*(t))\right\|_{Z'} +\|f\|_1.$$ However, by Proposition~\ref{prop:Norm1},
\begin{align*}
\|f\|_{Y_{Z'}}&\approx \sup_{\|S_I g\|_Z\leq 1}\int_0^1 \frac{I(t)}{t}(f^{**}(t)-f^*(t))g^*(t)\d t+\|f\|_1\\
	&\approx \sup_{\| g\|_Z\leq 1}\int_0^1 \frac{I(t)}{t}(f^{**}(t)-f^*(t))g^*(t)\d t+\|f\|_1\approx \|f\|_{Y_X}.
\end{align*}
\end{proof}


\subsection{Optimality and supremum operators}
This section presents characterization of optimal spaces by the boundedness properties of supremum operators $S_I$ and $T_I$. We require two preliminary results regarding the optimal domain norm.

Next proposition is an analogue of Theorem~\ref{cor:TIMIsBoundedOnTarget} for the operator $S_I$.

\begin{prop}\label{prop:domainInt}
Let $Y$ be an r.i.~space and assume that a quasiconcave function $I$ satisfies $\eqref{eq:00001}$ and the average property. Then the operator $S_I$ is bounded on $X_Y'$.
\end{prop}
\begin{proof}
    By virtue of Proposition~\ref{prop:OptimalDomain}, the optimal domain space under the map $H_I$, $X_Y$, exists for every r.i.~space $Y$ and is optimal in $$\|H_I f\|_{Y}\lesssim \|f\|_{X_Y},\quad f\in\mathcal M_+(0, 1).$$  This means, by duality, that $X_Y'$ is optimal in 
    \begin{equation}\label{eq:OptimalSpace1}
    \|R_If^*\|_{X_Y'}\lesssim \|f\|_{Y'},\quad f\in\mathcal M_+(0, 1).
    \end{equation} From the definition of the norm of the space $Y_{X_Y}'$, it easily follows that $Y_{X_Y}'$ is optimal in \begin{equation}\label{eq:OptimalSpace3}
        \|R_I f^*\|_{X_Y'}\lesssim \|f\|_{Y_{X_Y}'},\quad f\in\mathcal M_+(0, 1).
    \end{equation} Let us check that $X_Y'$, too, is optimal therein. Assume that 
\begin{equation*}
    \|R_I f^*\|_Z\lesssim \|f\|_{Y_{X_Y}'},\quad f \in\mathcal M_+(0, 1),
\end{equation*} for some r.i.~space $Z$. We need to show that $X_Y'\emb Z$. Optimality of $Y_{X_Y}'$ in $\eqref{eq:OptimalSpace3}$ and inequality $\eqref{eq:OptimalSpace1}$ tell us that $Y'\emb Y_{X_Y}'$, which in turn implies that $$\|R_I f^*\|_Z\lesssim \|f\|_{Y'}.$$ Finally, optimality of space $X_Y'$ in $\eqref{eq:OptimalSpace1}$ yields $X_Y'\emb Z$ as we wanted.

Now, let us set $\|f\|_Z=\|S_If\|_{X_Y'}$ for $f\in\mathcal M_+(0, 1)$. By virtue of Theorem~\ref{cor:ZRIQ}, $Z$ is an r.i.~space. Then, by $\eqref{eq:alphabetaOptimal}$, for the space $X_Y$ instead of $X$, we obtain $$\|R_I f^*\|_Z\lesssim \|f\|_{Y_{X_Y}'}.$$ Optimality of $X_Y'$ in $\eqref{eq:OptimalSpace3}$ therefore tells us that $\|f\|_Z\lesssim \|f\|_{X_Y'}$. In other words, $S_I$ is bounded on $X_Y'$.
\end{proof}

In the final theorem of the section, we will need an alternative description of
the optimal domain norm. Boundedness of the operator $T_I$ plays essential role
in~there.

\begin{thm}\label{thm:hvezdickadomeny}
    Let $I$ be a quasiconcave function satisfying $\eqref{eq:00001}$. Let $Y$ be an r.i.~space such that $H_I 1\in Y$. If $T_I$ is bounded on $Y'$, then 
$$\sup_{h\sim f}\|H_I h\|_Y\approx \|H_I f^*\|_Y,\quad f\in\mathcal M_+(0, 1).$$
In other words, for the optimal domain space we have $$\|f\|_{X_Y}\approx \|H_I f^*\|_Y,\quad f\in\mathcal M_+(0, 1).$$
\end{thm}
\begin{proof}
    First, as $H_I 1\in Y$, $\|\cdot\|_{X_Y}$ is an r.i.~norm by Proposition~\ref{prop:OptimalDomain}. Next, as $T_I$ is bounded on $Y'$, we have $$\|g\|_{Y'}\approx \|T_I g\|_{Y'},\quad g\in\mathcal M_+(0, 1).$$
    We need to check that $$\sup_{h\sim f}\|H_I h\|_Y\lesssim \|H_I f^*\|_Y,\quad f\in\mathcal M_+(0, 1).$$
    To this end, fix $f\in\mathcal M_+(0, 1)$. We estimate
    \begin{align*}
        \left\|\int_t^1 \frac{f(s)}{I(s)}\d s\right\|_Y&=\sup_{\|g\|_{Y'}\leq 1}\int_0^1 g^*(t)\int_t^1 \frac{f(s)}{I(s)}\d s\d t\\
        &=\sup_{\|g\|_{Y'}\leq 1}\int_0^1 \frac{f(s)}{I(s)}\int_0^s g^*(t) \d t\d s\\
        &\leq \sup_{\|g\|_{Y'}\leq 1}\int_0^1 \frac{f(s)}{I(s)}\int_0^s (T_I g)(t) \d t\d s\\
        &\approx \sup_{\|g\|_{Y'}\leq 1}\int_0^1  \frac{f(s)}{\int_0^s \frac{I(r)}{r}\d r}\int_0^s (T_I g)(t) \d t\d s\\
        &\leq \sup_{\|g\|_{Y'}\leq 1}\int_0^1 \frac{f^*(s)}{\int_0^s \frac{I(r)}{r}\d r}\int_0^s (T_I g)(t) \d t\d s\\
        &\approx \sup_{\|T_Ig\|_{Y'}\leq 1}\int_0^1 (T_I g)(t)\int_t^1 \frac{f^*(s)}{I(s)}\d s\d t\\
        &\approx \sup_{\|g\|_{Y'}\leq 1}\int_0^1 g^*(t)\int_t^1 \frac{f^*(s)}{I(s)}\d s\d t = \left\|\int_t^1 \frac{f^*(s)}{I(s)}\d s\right\|_Y.
    \end{align*} Here, the first and second “$\approx$” is the assumption that $I(t)\approx \int_0^t \frac{I(s)}{s}\d s$. In the second inequality we used the Hardy-Littlewood inequality $\eqref{eq:HLIneq}$, and the fact that $$s\mapsto \frac{1}{\int_0^s \frac{I(t)}{t}\d t}\int_0^s (T_I g)(t) \d t=\frac{1}{\int_0^s \frac{I(t)}{t}\d t}\int_0^s \sup_{t\leq r<1}\frac{r}{I(r)}g^*(r) \frac{I(t)}{t}\d t$$ is nonincreasing, being the integral mean of a nonincreasing function with respect to the measure $\frac{I(t)}{t}\d t$.
\end{proof}

We are finally prepared to prove the main theorem of the paper.

\begin{proof}[Proof of Theorem~\ref{thm:main}]
Since $\frac{1}{I}$ is integrable, the optimal domain space $X_Z$ exists for every r.i.~space $Z$ by Proposition~\ref{prop:OptimalDomain}. By Proposition~\ref{prop:domainInt}, $S_I$ is bounded on $X_Y'$. This, by Theorem~\ref{thm:Norm2}, means that 
\begin{equation}\label{eq:Optimality1}
\|f\|_{Y_{X_Y}}\approx\left\|\frac{I(t)}{t}(f^{**}(t)-f^*(t))\right\|_{X_Y}+\|f\|_1,\quad f\in\mathcal M_+(0, 1).
\end{equation} Now, given an r.i.~space $Z$, operator $T_I$ is bounded on $Y_Z'$ by virtue of Theorem~\ref{cor:TIMIsBoundedOnTarget}. Hence, by Theorem~\ref{thm:hvezdickadomeny}, we have
\begin{equation}\label{eq:Optimality2}
\|f\|_{X_{Y_Z}}\approx \left\|\int_t^1 \frac{f^*(s)}{I(s)}\d s\right\|_{Y_Z},\quad f\in\mathcal M_+(0, 1).
\end{equation}

Assuming now that $X$ is optimal for some $Y$, we have that $X=X_Y$. Consequently $\eqref{eq:Optimality1}$ turns into $\eqref{eq:Optimality0}$ and $S_I$ is bounded $X'$.

On the contrary, assuming $S_I$ is bounded on $X'$, we obtain validity of $\eqref{eq:Optimality0}$ by Theorem~\ref{thm:Norm2}. We show that $X_{Y_X}= X$. To this end, let $f\in\mathcal M_+(0, 1)$ be given. Then
\begin{equation}\label{eq:Optimality3}
\begin{split}
\|f\|_{X_{Y_X}}&\approx\left\|\int_t^1 \frac{f^*(s)}{I(s)}\d s\right\|_{Y_X}\\
	&\approx\left\| \frac{I(t)}{t}\left(\frac{1}{t}\int_0^t\int_s^1 \frac{f^*(r)}{I(r)}\d r \d s-\int_t^1 \frac{f^*(s)}{I(s)}\d s\right)\right\|_X+\int_0^1\int_t^1 \frac{f^*(s)}{I(s)}\d s \d t\\
&= \left\|\frac{I(t)}{t^{2}}\int_0^t\frac{r}{I(r)}f^*(r)\d r\right\|_X+\int_0^1\int_t^1 \frac{f^*(s)}{I(s)}\d s \d t,\\
\end{split}
\end{equation} because, by Fubini's theorem,
$$\frac{1}{t}\int_0^t\int_s^1 \frac{f^*(r)}{I(r)}\d r \d s-\int_t^1 \frac{f^*(s)}{I(s)}\d s=\frac{1}{t}\int_0^t\frac{s}{I(s)}f^*(s)\d s.$$
Further,
\begin{equation}\label{eq:Optimality4}
\begin{split}
\int_0^1\int_t^1 \frac{f^*(s)}{I(s)}\d s \d t&\leq \int_0^1 f^*(t)\int_t^1\frac{\d s}{I(s)}\d t\leq \int_0^1 f^*(t)\d t\int_0^1 \frac{\d s}{I(s)}\lesssim \|f\|_1\lesssim \|f\|_X.
\end{split}
\end{equation}
Observe that the operator $R$ defined by
$$Rf(t)=\frac{I(t)}{t^{2}}\int_0^t \frac{s}{I(s)}f(s)\d s,\quad f\in\mathcal M_+(0, 1), t\in (0, 1),$$ is associate to the operator $H$ defined in $\eqref{eq:OperatorH}$. Now, quasiconcavity of $I$ implies that $$\int_0^t \frac{s}{I(s)}\d s\approx \frac{t^2}{I(t)},\quad t\in(0, 1).$$ Therefore, for the operator $R'$ defined by $$R'f(t)=\frac{1}{\int_0^t\frac{s}{I(s)}\d s}\int_0^t \frac{s}{I(s)}f(s)\d s,\quad f\in\mathcal M_+(0, 1), t\in (0, 1),$$ we have $$Rf(t)\approx R'f(t),\quad f\in\mathcal M_+(0, 1), t\in (0, 1).$$ The advantage of the operator $R'$ is that $R'f^*$ is nonincreasing for every $f\in\mathcal M_+(0, 1)$, being the integral mean of $f^*$ with respect to the measure $\frac{t}{I(t)}\d t$. Hence, using $\eqref{eq:Hjeleqthg}$, we have

\begin{align*}
    \|Rf^*\|_X&\approx \|R'f^*\|_X=\sup_{\|g\|_{X'}\leq 1}\int_0^1 R'f^*(t)g^*(t)\d t\approx\sup_{\|g\|_{X'}\leq 1}\int_0^1 Rf^*(t)g^*(t)\d t\\
    &=\sup_{\|g\|_{X'}\leq 1}\int_0^1 f^*(t)Hg^*(t)\d t
    \lesssim \sup_{\|g\|_{X'}\leq 1}\int_0^1 f^*(t)g^*(t)\d t=\|f^*\|_X,
\end{align*} for every $f\in\mathcal M_+(0, 1)$.

We also see that
\begin{align*}
\frac{I(t)}{t^{2}}\int_0^t\frac{r}{I(r)}f^*(r)\d r\geq f^*(t)\cdot\frac{I(t)}{t^2}\int_0^t \frac{r}{I(r)}\d r\geq f^*(t)\cdot\frac{I(t)}{t^2}\int_{\frac{t}{2}}^t \frac{r}{I(r)}\d r\approx f^*(t),
\end{align*} as $I\in\Delta_2$.
Thus $$\left\|\frac{I(t)}{t^{2}}\int_0^t\frac{r}{I(r)}f^*(r)\d r\right\|_X\approx \|f\|_X,\quad f\in\mathcal M_+(0, 1).$$ Combining this with $\eqref{eq:Optimality3}$ and $\eqref{eq:Optimality4}$ yields
$$\|f\|_{X_{Y_X}}\approx \|f\|_X.$$

Turning our attention to the second part, the assumption on optimality of $Y$ for some space $X$ implies $Y=Y_X$. Hence $T_I$ is bounded on $Y'$ and $\eqref{eq:Optimality2}$ turns into $\eqref{eq:Optimality01}$.

In the other direction, let us assume that $T_I$ is bounded on $Y'$. We show that $Y=Y_{X_Y}$, which shows optimality of $Y$ for some space -- $X_Y$ in fact. 

For $t\in (0, \frac{1}{2})$ we estimate
\begin{equation*}
\begin{split}
\int_t^1\frac{I(s)}{s^{2}}(f^{**}(s)-f^*(s))\d s&=\int_t^1 \frac{I(s)}{s^{3}}\int_0^s f^*(r)\d r\d s-\int_t^1\frac{I(s)}{s^{2}}f^*(s)\d s\\
	&=\int_t^1 \frac{I(s)}{s^{3}}\int_0^t f^*(r)\d r\d s+\int_t^1\frac{I(s)}{s^{3}}\int_t^s f^*(r)\d r\d s-\int_t^1\frac{I(s)}{s^{2}}f^*(s)\d s\\
	&\geq \int_t^1 \frac{I(s)}{s^{3}}\int_0^tf^*(r)\d r\d s+\int_t^1\frac{I(s)}{s^{3}}(s-t)f^*(s)\d s -\int_t^1\frac{I(s)}{s^{2}}f^*(s)\d s\\
	&=\int_t^1\frac{I(s)}{s^{3}}\int_0^t f^*(r)\d r\d s-t\int_t^1\frac{I(s)}{s^{3}}f^*(s)\d s\\
	&\geq t\int_t^1\frac{I(s)}{s^3}f^{**}(t)\d s-tf^*(t)\int_t^1 \frac{I(s)}{s^3}\d s\\
	&\geq t\int_t^{2t}\frac{I(s)}{s^{3}}\d s(f^{**}(t)-f^*(t))\approx \frac{I(t)}{t}(f^{**}(t)-f^*(t)).
 \end{split}
 \end{equation*}
Using this, $\eqref{eq:Optimality1}$ and Theorem~\ref{thm:hvezdickadomeny}, we have
\begin{equation}\label{eq:Optimality5}
\begin{split}
\|f\|_{Y_{X_Y}}&\approx\left\|\frac{I(t)}{t}(f^{**}(t)-f^*(t))\right\|_{X_Y}+\|f\|_1\\
	&\lesssim \left\|\int_t^1\frac{I(s)}{s^{2}}(f^{**}(s)-f^*(s))\d s\right\|_{X_Y}+\|f\|_1\\
	&\approx \left\|\int_t^1 \frac{1}{I(s)}\int_s^1\frac{I(r)}{r^{2}}(f^{**}(r)-f^*(r))\d r\d s\right\|_Y+\|f\|_1.
\end{split}
\end{equation} Now, for every $h\in\mathcal M_+(0, 1)$, we have
\begin{equation}\label{eq:7894616}
\begin{split}
\int_t^1 \frac{1}{I(s)}\int_s^1 h(r)\frac{{\rm d} r}{r}\d s&=\int_t^1\frac{h(r)}{r}\int_t^r \frac{{\rm d}s}{I(s)}\d r\leq \int_t^1\frac{h(r)}{r}\int_0^r \frac{{\rm d}s}{I(s)}\d r\\
	&\lesssim \int_t^1\frac{h(r)}{r}\cdot\frac{r}{I(r)}\d r=\int_t^1\frac{h(r)}{I(r)}\d r,
\end{split}
\end{equation} where the first inequality is the average property of $I$.
Further,
\begin{equation}\label{eq:54566}
\begin{split}
\int_t^1 \frac{1}{s^2}\int_0^s h(r)\d r\d s&=\int_t^1\frac{1}{s^2}\int_0^t h(y)\d y\d s+\int_t^1\frac{1}{s^2}\int_t^sh(r)\d r\d s\\
	&\leq \int_0^t h(y)\d y\int_t^1 \frac{{\rm d} s}{s^2}+\int_t^1 h(r)\int_r^1\frac{{\rm d} s}{s^2}\d r\\
 &\leq \frac{1}{t}\int_0^t h(r)\d r+\int_t^1 \frac{h(r)}{r}\d r.
\end{split}
\end{equation} Deploying $\eqref{eq:7894616}$ for $h(r)=\frac{I(r)}{r}(f^{**}(r)-f^*(r))$ and $\eqref{eq:54566}$ for $h(r)=f^*(r)$ in this order in $\eqref{eq:Optimality5}$ yields
\begin{equation}\label{eq:koneccc}
\begin{split}
\|f\|_{Y_{X_Y}}&\lesssim\left\|\int_t^1\frac{1}{s}(f^{**}(s)-f^*(s))\d s\right\|_Y+\|f\|_1\\
	&=\left\|\int_t^1\frac{1}{s^2}\int_0^s f^*(r)\d r\d s-\int_t^1f^*(s)\frac{{\rm d} s}{s}\right\|_Y+\|f\|_1\\
	&\leq \|f^{**}\|_Y+\|f\|_1.
\end{split}
\end{equation} 
We now claim that $T_I\colon Y'\to Y'$ implies $Y\subset \Lambda_I$. In the proof of Proposition~\ref{lem:INTTargetSpace} we showed that $\Lambda_I'=m_{\widetilde I}$ whenever $I$ satisfies $\eqref{eq:00001}$. Therefore, we can equivalently show that $Y'\supset m_{\widetilde I}$. Now, for every $f\in m_{\widetilde I}$ there exists $c\geq 0$ such that $f^*(t)\leq c\frac{I(t)}{t}, t\in (0, 1)$. Thus, from the lattice property of the space $Y'$, it suffices to show that $t\mapsto \frac{I(t)}{t}\in Y'$. We observe that $(T_I 1)(t)=\frac{I(t)}{t}, t\in (0, 1)$, and so, using the boundedness of the operator $T_I$ on $Y'$, we have 
$$
\left\|\frac{I(t)}{t}\right\|_{Y'}=\|T_I 1\|_{Y'}\lesssim \|1\|_{Y'}<\infty.
$$ 
Hence, combining Proposition~\ref{lem:INTTargetSpace} with ~\cite[Theorem 10.3.12]{FS} (note that ~\eqref{eq:BoundedMaximalOp} is satisfied), we conclude that $f\mapsto f^{**}$ is bounded on $Y$. Adding this piece of information to $\eqref{eq:koneccc}$, we obtain $$\|f\|_{Y_{X_Y}}\lesssim \|f\|_Y.$$ As $Y_{X_Y}\subset Y$ holds trivially, the proof is complete.
\end{proof}


\section{Applications to Sobolev embeddings}\label{SobEmb}
In this section we come back to conditions which we used throughout the paper and see which functions satisfy them. As the function $I$ plays the role of the isoperimetric profile of a domain, we will be mainly interested in two types of domains. Firstly, those whose isoperimetric profile is related to polynomials $I(t)=t^\alpha, \alpha\in \left[\frac{1}{n'}, 1\right)$ and, secondly, product probability spaces. To this end, we check which conditions are satisfied. We have already stated in the preliminary section that polynomials belong to class $\mathcal Q$.

We shall first summarize relevant observations concerning polynomials.
\begin{prop}[Polynomials $t^\alpha$ for $\alpha \in (0, 1)$]\label{prop:polynomials}
    Let $\alpha\in (0, 1)$ be given. Then $I(t)=t^\alpha$ belongs to class $\mathcal Q$. To be precise, for the constants $c$ and $d$ from the Definition~\ref{def:Conditions} we have $c=\frac{1}{2-\alpha}$ and $d=\frac{1}{1-\alpha}$. Consequently, the conclusion of Theorem~\ref{thm:main} is valid for such a choice of function $I$.
\end{prop}

We now show that the isoperimetric function of product probability spaces satisfy the condition $\eqref{eq:00001}$. Let us consider $\Phi\colon[0, \infty)\to [0, \infty)$. Assume $\Phi$ is twice differentiable, strictly increasing and convex in $(0, \infty)$, $\sqrt{\Phi}$ is concave  and $\Phi(0)=0$. Define further a measure on $\Rea$ by 
\begin{equation*}
\d \mu_\Phi (x)=c_\Phi e^{-\Phi(\abs{x})}\d x,
\end{equation*} where $c_\Phi$ is such that $\mu_\Phi(\Rea)=1$.

We then define its \emph{product measure} $\mu_{\Phi, n}$ on $\Rea^n$ as$$\mu_{\Phi, n}=\underbrace{\mu_\Phi\times\dots \times \mu_\Phi}_{n\text{-times}}.$$

Then $(\Rea^n, \mu_{\Phi, n})$ is a probability space.

It is further known by \cite[Chapter 7]{CiPiSl} that
\begin{equation}\label{eq:ProbSpaceIsoFnction}
    I(t)=I_{\Rea^n, \mu_{\Phi, n}}(t)\approx t\Phi'\left(\Phi^{-1}\left(\log\frac{2}{t}\right)\right),\quad t\in \left(0, \frac{1}{2}\right).
\end{equation} 
Note that $I$ is quasiconcave function -- $I$ being nondecreasing is proved in \cite[Lemma 11.1]{CiPiSl}, while $t\mapsto \frac{I(t)}{t}$ being nonincreasing follows from the fact that $\Phi$ is increasing and convex. 
We now check that $I'(t)\approx \frac{I(t)}{t}$ on some right neighbourhood of $0$. This implies that $I$ satisfies condition $\eqref{eq:00001}$.

First, we observe that $\Phi\in\Delta_2$. Indeed, to see this, recall that $\sqrt{\Phi}$ is concave and so $\sqrt{\Phi}\in\Delta_2$. We estimate $$\Phi(2t)=\left(\sqrt{\Phi(2t)}\right)^2\leq \left(c\sqrt{\Phi(t)}\right)^2=c^2\Phi(t),\quad t\in (0, \infty).$$

As $\Phi$ is an increasing convex function with $\Phi(0)=0$, we write $$\Phi(t)=\int_0^t\Phi'(s)\d s,\quad t\in (0, \infty).$$ Convexity tells us that $\Phi'$ is nondecreasing and so, for every $t\in (0, \infty)$, we have $$\Phi(t)=\int_0^t \Phi'(s)\d s\leq t\Phi'(t)\leq \int_t^{2t}\Phi'(s)\d s\leq \Phi(2t).$$
Combining with the knowledge that $\Phi\in\Delta_2$, we obtain $$\Phi'(t)\approx \frac{\Phi(t)}{t},\quad t\in (0, \infty).$$ Plugging this new information into $\eqref{eq:ProbSpaceIsoFnction}$, we have 
\begin{equation}\label{eq:ProbSpace2}
I(t)\approx \frac{t\log\frac{2}{t}}{\Phi^{-1}\left(\log\frac{2}{t}\right)},\quad t\in \left(0, \frac{1}{2}\right),
\end{equation} and so we may as well consider the very last expression to be the representative of $I$. Differentiating, we get 
\begin{align*}
    I'(t)&= \frac{\left(\log\frac{2}{t}-1\right)\Phi^{-1}\left(\log\frac{2}{t}\right)-\frac{1}{\Phi'\left(\Phi^{-1}\left(\log\frac{2}{t}\right)\right)}\cdot \frac{-1}{t}\cdot t\log\frac{2}{t}}{\left(\Phi^{-1}\left(\log\frac{2}{t}\right)\right)^2}\\
    &=\underbrace{\frac{\left(\log\frac{2}{t}-1\right)}{\Phi^{-1}\left(\log\frac{2}{t}\right)}}_{A(t)}+\underbrace{\frac{\log\frac{2}{t}}{\Phi'\left(\Phi^{-1}\left(\log\frac{2}{t}\right)\right)\cdot \left(\Phi^{-1}\left(\log\frac{2}{t}\right)\right)^2}}_{B(t)},\quad t\in \left(0, \frac{1}{2}\right).
\end{align*} As for $A$, we simply note that $\log\frac{2}{t}-1\approx \log\frac{2}{t}$ on some right neighbourhood of $0$, because $\lim_{t\to 0^+}\log\frac{2}{t}=\infty.$ Therefore, it follows that $A(t)\approx \frac{I(t)}{t}$. Using both $\eqref{eq:ProbSpaceIsoFnction}$ and $\eqref{eq:ProbSpace2}$, we obtain $$B(t)\approx \frac{\log\frac{2}{t}}{\frac{I(t)}{t}}\cdot \left(\frac{I(t)}{t\log\frac{2}{t}}\right)^2=\frac{I(t)}{t\log\frac{2}{t}}.$$ Finally, observing that $\lim_{t\to 0^+} \frac{B(t)}{A(t)}=0$, we conclude that $B$ is negligible.

We have therefore proved the following proposition.
\begin{prop}\label{Prop:PPSI}
    Let $I$ be as in $\eqref{eq:ProbSpaceIsoFnction}$. Then $I$ satisfies condition $\eqref{eq:00001}$.
\end{prop}


Recall that Maz'ya classes of domains $\mathcal J_\alpha$ for $\alpha\in\left[\frac{1}{n'}, 1\right)$ are defined as $$\mathcal J_\alpha=\left\{\Omega : I_{\Omega}(t)\gtrsim t^\alpha, t\in \left[0, \frac{1}{2}\right]\right\}.$$
Since $I(t)=t^\alpha$ enjoys the average property, we have, by virtue of \cite[Proposition 8.6]{CiPiSl},  $$\|R_I^m f\|_X\approx \left\|\frac{t^{m-1}}{I(t)^m}\int_0^t f(s)\d s\right\|_X,\quad f\in\mathcal M_+(0, 1),$$
and $$\|H_I^m f\|_X\approx \left\|\int_t^1 \frac{s^{m-1}}{I(s)^m}f(s)\d s\right\|_X,\quad f\in\mathcal M_+(0, 1),$$ for every r.i.~space $X$. Define a function $J\colon (0, 1)\to (0, \infty)$ by  $$J(t)=\frac{I(t)^m}{t^{m-1}}=t^{1-m(1-\alpha)},\quad t\in (0, 1).$$ From here, we see that whenever $1-m(1-\alpha)>0$, then $J(t)$ is an increasing, strictly concave bijection of $(0, 1)$ onto itself. Proposition~\ref{prop:polynomials} therefore asserts that the function $J\in\mathcal Q$. Therefore, for this particular choice of $I$, Theorem~\ref{thm:main} reads as follows:

\begin{thm}\label{cor:HI&SO}
    Let $\alpha\in \left[\frac{1}{n'}, 1\right)$ be given and let $m\in\Nat$ be such that $1-m(1-\alpha)>0$. Put $J(t)=t^{1-m(1-\alpha)}$ for $t\in (0, 1)$. Then an r.i.~space $X$ is the optimal domain space under the map $H_I^m$ for some r.i.~space $Y$ if and only if $S_J$ is bounded on $X'$. In that case,
\begin{equation*}
\|f\|_{Y_X}\approx \left\|\frac{J(t)}{t}(f^{**}(t)-f^*(t))\right\|_X+\|f\|_1,\quad f\in\mathcal M_+(0, 1).
\end{equation*}

Vice versa, an r.i.~space $Y$ is the optimal target space  under the map $H_I^m$ for some r.i.~space $X$ if and only if $T_J$ is bounded on $Y'$. In that case,
\begin{equation*}
\|f\|_{X_Y}\approx \left\|\int_t^1 \frac{f^*(s)}{J(s)}\d s\right\|_{Y},\quad f\in\mathcal M_+(0, 1).
\end{equation*}
\end{thm}

Recall that a bounded open $\Omega\subset\Rea^n$ is said to be a \emph{John domain}, if there exists $x_0\in\Omega$ and $c>0$ such that for every $x\in\Omega$ there exists a rectifiable curve $\gamma_x\colon [0, l]\to \Omega$, parametrized by its arclength, such that $${\rm dist}(\gamma_x(t), \partial\Omega)\geq ct,\quad t\in [0, l].$$ It is known that for every John domain $\Omega$ we have $I_\Omega (t)\approx t^\frac{1}{n'}$. Note that every Lipschitz domain is also a John domain. 

Now, given $\Omega\subset\Rea^n$ a John domain, \cite[Theorem 6.1]{CiPiSl} asserts that $$H_{I_\Omega}^m\colon X\to Y\iff V^mX\to Y$$ whenever $X$ and $Y$ are r.i.~spaces. Combining this with Theorem~\ref{cor:HI&SO}, we obtain the following.
\begin{thm}\label{thmmasd}
    Let $\Omega\subset\Rea^n$ be a John domain. Let $m\in\Nat, m<n$ be given and put $J(t)=t^{1-\frac{m}{n}}$ for $t\in (0, 1)$. Then an r.i.~space $X(\Omega)$ is the optimal domain space in the $m$-th order Sobolev embedding for some r.i.~space $Y(\Omega)$ if and only if $S_J$ is bounded on $X'$. In that case, $$\|f\|_{Y_X}\approx \left\|t^{-\frac{m}{n}}(f^{**}(t)-f^*(t))\right\|_X+\|f\|_1,\quad f\in\mathcal M_+(0, 1).$$

    Vice versa, an r.i.~space $Y(\Omega)$ is the optimal target space in the $m$-th order Sobolev embedding for some r.i.~space $X(\Omega)$, if and only if $T_J$ is bounded on $Y'$, in which case $$\|f\|_{X_Y}\approx \left\|\int_t^1 s^{\frac{m}{n}-1}f^*(s)\d s\right\|_{Y},\quad f\in\mathcal M_+(0, 1).$$
\end{thm}

Theorem~\ref{thmmasd} recovers \cite[Theorem A]{KePi2009} for Lipschitz domains and extends the result to John domains.

Now, the special case of \cite[Section 5.3.3]{Ma11} exhibits existence of a particular domain $\Omega_\alpha$ whose isoperimetric profile satisfies $I_{\Omega_\alpha}(t)\approx t^\alpha$. Let $\alpha\in \left[ \frac{1}{n'}, 1\right)$ and define $\eta_\alpha\colon \left[0 ,\frac{1}{1-\alpha}\right]\to [0, \infty)$ by $$\eta_\alpha(r)=\omega_{n-1}^{-\frac{1}{n-1}}(1-(1-\alpha)r)^\frac{\alpha}{(1-\alpha)(n-1)},\quad r\in \left[0, \frac{1}{1-\alpha}\right],$$ where $\omega_{n-1}$ denotes the Lebesgue measure of the unit ball in $\Rea^{n-1}$. \\Define $\Omega_\alpha\subset\Rea^n$ as 
\begin{equation}\label{eq:MazyaDomain}
    \Omega_\alpha=\left\{(x', x_n)\in\Rea^n: x'\in\Rea^{n-1}, 0 < x_n<\frac{1}{1-\alpha}, \abs{x'}<\eta_\alpha(x_n)\right\}.
\end{equation} Then $\abs{\Omega_\alpha}=1$ and $I_{\Omega_\alpha}(t)\approx t^\alpha, t\in \left[0, \frac{1}{2}\right]$. In particular, $\Omega_\alpha\in\mathcal J_\alpha$.

\begin{thm}\label{cor:MazSob}
    Let $\alpha\in\left[\frac{1}{n'}, 1\right)$ and $m\in\Nat$ be such that $1-m(1-\alpha)>0$. Put $J(t)=t^{1-m(1-\alpha)}$ and let $\Omega_\alpha$ be as in $\eqref{eq:MazyaDomain}$. Then an r.i.~space $X(\Omega_\alpha)$ is the optimal domain space in the $m$-th order Sobolev embedding for some r.i.~space $Y(\Omega_\alpha)$ if and only if $S_J$ is bounded on $X'$. In this case, $V^m X(\Omega)\to Y(\Omega)$ for every $\Omega\in\mathcal J_\alpha$ and $$\|f\|_{Y_X}\approx \|t^{-m(1-\alpha)}(f^{**}(t)-f^*(t))\|_X+\|f\|_1,\quad f\in\mathcal M_+(0, 1).$$

    Vice versa, an r.i.~space $Y(\Omega_\alpha)$ is the optimal target space in the $m$-th order Sobolev embedding for some r.i.~space $X(\Omega_\alpha)$ if and only if $T_J$ is bounded on $Y'$. In this case, $V^m X(\Omega)\to Y(\Omega)$ for every $\Omega\in\mathcal J_\alpha$ and $$\|f\|_{X_Y}\approx \left\|\int_t^1 s^{m(1-\alpha)-1}f^*(s)\d s\right\|_X,\quad f\in\mathcal M_+(0, 1).$$
\end{thm}
\begin{proof}
    By \cite[Theorem 6.4]{CiPiSl} $H_J\colon X\to Y$ if and only if $\forall \Omega\in \mathcal J_\alpha\colon V^m X(\Omega)\to Y(\Omega)$. Hence the result follows from Theorem~\ref{thm:main}.
\end{proof}

We will now present a specification of our main results to Sobolev embeddings involving Lorentz--Zygmund spaces to obtain a new way to compute the norm and compare it to the known results.

\begin{thm}
     Let $\alpha\in \left[\frac{1}{n'}, 1\right)$ be given and let $m\in\Nat$ be such that $1-m(1-\alpha)>0$. Set $X=L^{p, q, \beta}$. Then for the optimal target space of $X$ in the Sobolev embedding we have
     \begin{equation}\label{eq:GLZ}
         \|f\|_{Y_X}\approx\left(\int_0^1 t^{\frac{q}{p}-1}\ell_1^{\beta q}(t)\left[\left(s\mapsto s^{-m(1-\alpha)}(f^{**}(s)-f^*(s))\right)^*(t)\right]^q\d t\right)^\frac{1}{q}+\|f\|_1.
     \end{equation} Furthermore, the norm in $\eqref{eq:GLZ}$ is equivalent to
     \begin{align*}
         &\left(\int_0^1 t^\frac{q-qpm(1-\alpha)-p}{p}\ell_1^{\beta q}(t)(f^*)^q(t)\d t\right)^\frac{1}{q}\quad \text{if }  p=q=1, \beta\geq0 \text{ or }p\in \left(1, \frac{1}{m(1-\alpha)}\right),\\
         &\left(\int_0^1 t^{-1}\ell_1^{q(\beta-1)}(t)(f^*)^q(t)\d t\right)^\frac{1}{q}\quad \text{if } p=\frac{1}{m(1-\alpha)}, \beta<1-\frac{1}{q},\\
         &\left(\int_0^1 t^{-1}\ell_1^{-1}(t)\ell_2^{-q}(t)(f^*)^q(t)\d t\right)^\frac{1}{q}\quad \text{if } p=\frac{1}{m(1-\alpha)}, q\in(1, \infty], \beta=1-\frac{1}{q},\\
         &f^*(0+)\quad \text{if } p=\frac{1}{m(1-\alpha)}, \beta>1-\frac{1}{q} \text{ or } p=\frac{1}{m(1-\alpha)}, q=1, \beta\geq0 \text{ or } p>\frac{1}{m(1-\alpha)}.
     \end{align*}
\end{thm}
\begin{proof}
    The relation $\eqref{eq:GLZ}$ is an immediate application of Theorem~\ref{cor:MazSob}. The equivalence can be derived from \cite[Proposition 6.4]{Mi}.
\end{proof}

\begin{rem}
    We would like to point out that, although the equivalence does not seem surprising, relations between spaces involving the quantity $f^{**}-f^*$ to those involving just $f^*$ are quite difficult, and even the question whether such spaces are linear is highly nontrivial, for details see, for example, \cite{CaGoMaPi}.
\end{rem}

\section*{Acknowledgments.}
This research was supported by grant no.~23-04720S of the Czech Science Foundation.

\bibliographystyle{dabbrv}
\bibliography{bibliography}

\end{document}